\documentclass[11pt]{article}

\usepackage[T1]{fontenc}
\usepackage{lmodern}
\usepackage[margin=1in]{geometry}
\usepackage{amsmath,amssymb,amsthm,mathtools,mathrsfs}
\usepackage{enumitem}
\usepackage[hidelinks]{hyperref}
\usepackage[nameinlink,noabbrev]{cleveref}
\usepackage{microtype}

\newtheorem{theorem}{Theorem}[section]
\newtheorem{proposition}{Proposition}[section]
\newtheorem{lemma}{Lemma}[section]
\newtheorem{corollary}{Corollary}[section]
\theoremstyle{remark}
\newtheorem{remark}{Remark}[section]
\newtheorem{defn}{Definition}[section]
\newtheorem{assumption}{Assumption}[section]

\newcommand{\Ch}{\mathrm{Ch}}
\newcommand{\Gammak}{\Gamma_k}
\newcommand{\ddc}{dd^c}

\newcommand{\tr}{\operatorname{tr}}
\newcommand{\Lip}{\operatorname{Lip}}
\newcommand{\osc}{\operatorname{osc}}
\newcommand{\Vol}{\operatorname{Vol}}
\newcommand{\lintr}{\mathscr F}

\newcommand{\one}{\mathbf 1}
\newcommand{\R}{\mathbb R}
\newcommand{\C}{\mathbb C}

\newcommand{\ddbar}{\sqrt{-1}\,\partial\bar\partial}

\title{Regularization and H\"older continuity for complex
Hessian equations on Hermitian manifolds}
\author{Yulun Xu}
\date{\today}

\begin{document}

\maketitle

\begin{abstract}
Let $(M,\omega)$ be a compact Hermitian manifold, and let $\Gamma$ be a symmetric
convex cone. We develop a quantitative regularization method for
$\Gamma$-admissible functions. As an application, we prove H\"older continuity
for every bounded pluripotential solution of complex $m$-Hessian equations whose right-hand sides belong to
$L^p$, for $p>\frac{n}{m}$. These results extend to a more general class of
complex Hessian equations satisfying the structural condition used by
Guo--Phong--Tong \cite{GuoPhongTong2023}.
\end{abstract}

\section{Introduction}
The main goal of this paper is to address the long-standing problem of
H\"older regularity for solutions of complex Hessian equations on Hermitian
manifolds with right-hand sides in $L^p$. For such data, H\"older continuity
is in general the natural regularity one expects; see
\cite{DinewPlisZhang2019}.

The H\"older regularity theory of complex Hessian equations with right-hand sides in $L^p$ has two
logically different components.  One needs, first, a quantitative
regularization that respects the relevant positivity cone and, second, a
stability estimate that compares the regularized function with the solution
of the equation.  For the complex Monge--Amp\`ere operator, the regularization
theory of plurisubharmonic functions provides the first component.  On a
curved background, Demailly's regularization initially produces curvature
errors, which can be attenuated by a Kiselman--Legendre transform; see
\cite{Demailly1992,Demailly1994}. B{\l}ocki and Ko{\l}odziej \cite{BK}
defined local approximations for bounded plurisubharmonic functions and showed
that they can be patched together, thereby simplifying Demailly's argument in
the bounded setting. H\"older estimate was first established by Kolodziej \cite{Kolodziej2008}.  The H\"older exponent was refined in
\cite{DemaillyEtAl2014}. This regularity theory was extended to compact
Hermitian manifolds in \cite{KolodziejNguyen2019,LuPhungTo2021}.

However, for complex $m$-Hessian equations on compact K\"ahler manifolds,
H\"older regularity was largely open, according to Ko{\l}odziej--Nguyen
\cite{KolodziejNguyen2025}. The pluripotential theory of $m$-subharmonic functions,
together with basic $L^p$ and stability estimates, was developed in
\cite{Blocki2005,DinewKolodziej2014}. A central obstacle to H\"older estimates
was the absence of a quantitative global regularization on a general compact
Hermitian manifold. Cheng and the author overcame this obstruction on compact
K\"ahler manifolds with nonnegative holomorphic bisectional curvature by using
a geometric sup-convolution \cite{ChengXu2024}. In the present paper, we
construct a quantitative regularization framework on arbitrary compact
Hermitian manifolds for a class of fully nonlinear spectral operators that
includes the complex $m$-Hessian operators.

Guo--Phong--Tong \cite{GuoPhongTong2023} introduced an auxiliary
Monge--Amp\`ere method for uniform $L^\infty$ estimates on compact K\"ahler
manifolds. Guo--Phong \cite{GuoPhong2024} subsequently developed this circle
of ideas for fully nonlinear equations, including the non-K\"ahler setting,
and emphasized its role as the beginning of a De Giorgi--Nash--Moser theory
for fully nonlinear equations. The
auxiliary Monge--Amp\`ere method is an essential input in the stability
argument below. It also underlies modulus-of-continuity estimates for complex
Monge--Amp\`ere equations on K\"ahler manifolds
\cite{GuoPhongTongWang2021} and on Hermitian manifolds \cite{Liu2025}. From
this perspective, the H\"older estimates proved here constitute a natural next 
step for fully nonlinear complex equations with $L^p$ data.

Let \((M,\omega)\) be a connected compact Hermitian manifold of complex
dimension \(n\), where
\[
 \omega=\sqrt{-1}\,g_{j\bar k}\,dz^j\wedge d\bar z^k.
\]
Let $\chi$ be a smooth real $(1,1)$-form on $M$.  For a real-valued
function $u$, set
\[
 A_u=g^{-1}(\chi+\ddc u),
 \qquad
 \lambda[\chi+\ddc u]:=\lambda(A_u).
\]
Thus, in local coordinates,
$(A_u)^i{}_j=g^{i\bar k}(\chi_{j\bar k}+u_{j\bar k})$.

Let $f\in C^\infty(\Gamma)\cap C(\overline\Gamma)$ be symmetric.  We use
the following assumptions on $(f,\Gamma)$.

\begin{assumption}\label{ass1}
\begin{enumerate}[label=\textup{(\arabic*)}]
\item \(\Gamma\subset\R^n\) is an open symmetric convex cone satisfying
\[
 \Gamma_n:=\{\lambda_i>0\}
 \subset\Gamma\subset
 \Gamma_1:=\left\{\sum_{i=1}^n\lambda_i>0\right\}.
\]
\item $f_i:=\partial f/\partial\lambda_i>0$ for $1\leq i\leq n$;
      $f$ is concave and positive on $\Gamma$, and
      $f=0$ on $\partial\Gamma$.
\end{enumerate}
\end{assumption}

\begin{assumption}\label{ass2}
$f$ is homogeneous of degree one.
\end{assumption}

\begin{assumption}\label{ass3}
For the spectral function $F(A)=f(\lambda(A))$, there is a
constant $\gamma_0>0$ such that, in every unitary frame,
\[
 \det\!\left(\frac{\partial F}
                   {\partial A_{i\bar j}}(A)\right)
 \geq\gamma_0
 \qquad\text{whenever }\lambda(A)\in\Gamma.
\]
\end{assumption}

\begin{remark}
Assumption~\ref{ass3} is the same as the condition used in
\cite{GuoPhongTong2023}. Many operators satisfy this assumption, including the
following:
\begin{enumerate}
\item The complex $k$-Hessian operator:
$f(\lambda)=\sigma_k^{1/k}(\lambda)$, $1\leq k\leq n$, with
$\Gamma_k:=\{\lambda:\sigma_i(\lambda)>0,\ 1\leq i\leq k\}$.
\item The $\sigma_k$-operator for $(n-1)$-plurisubharmonic functions:
$f(\lambda)=\sigma_k^{1/k}(\widetilde{\lambda})$, where
$\widetilde{\lambda}_i=\frac{1}{n-1}\sum_{j\neq i}\lambda_j$.
\item The $p$-fold sum operator:
$f(\lambda)=\bigl(\prod_{|J|=p}\lambda_J\bigr)^{1/N}$, where
$\lambda_J=\lambda_{j_1}+\lambda_{j_2}+\cdots+\lambda_{j_p}$ and
$N=\binom{n}{p}$.
\end{enumerate}
The second example with $k=n$ was studied by Tosatti--Weinkove \cite{TW1},
whereas the third was considered by Harvey--Lawson \cite{HL2} in connection
with $p$-geometry and $p$-potential theory. Moreover, a general result of
Gurvits \cite{Gur} (see also \cite{HL}) shows that all
G{\aa}rding--Dirichlet operators satisfy this assumption.

Complex quotient equations do not satisfy Assumption~\ref{ass3}. By
Cheng and the author \cite{ChengXu2026}, stability estimates for complex quotient
equations require greater regularity of the right-hand side. Consequently,
the method developed in the present paper does not directly yield H\"older
estimates for complex Hessian quotient equations with merely $L^p$
right-hand sides.
\end{remark}

\subsection{H\"older estimate}

Let \(u\in C(M)\) be a bounded viscosity solution of
\begin{equation}
F(A_u)=f(\lambda[\chi+\ddc u])=e^G>0,
 \qquad
 \sup_M u=-1, \qquad \lambda[\chi+dd^c u]\in \Gamma,
 \label{eq:main}
\end{equation}
where $G$ is continuous.
\begin{theorem}
\label{thm:main}
Let $(M,\omega)$ be a connected compact Hermitian manifold.  Suppose that
$(f,\Gamma)$ satisfies Assumptions~\ref{ass1} and~\ref{ass3}, and that $\chi$ is a smooth $(1,1)$-form with
$\lambda(g^{-1}\chi)\in\Gamma$ pointwise on $M$. Let $u$ be a bounded
viscosity solution of \eqref{eq:main}. Assume that
$e^{nG}\in L^{p_0}(M,\omega^n)$ for some $p_0>1$, and set
$q_0=p_0/(p_0-1)$. Then, for every
\begin{equation}
 0<\mu<\frac{2}{1+nq_0}
 \label{eq:target-exponent}
\end{equation}
one has
\[
 \|u\|_{C^{\mu}(M)}\leq C.
\]
Here $C$ depends only on $M,\omega,\chi$, $n,p_0,\mu$, the
structural data of $(f,\Gamma)$, and
$\|e^{nG}\|_{L^{p_0}(M,\omega^n)}$.
\end{theorem}

\begin{remark}
The range of H\"older exponent in Theorem \ref{thm:main} coincides with Lu-Phung-To \cite{LuPhungTo2021} in the complex Monge-Amp\'ere equation case.
\end{remark}

For pluripotential solutions of complex $m$-Hessian equations with right-hand sides
in $L^{p_0}$, we establish H\"older regularity without assuming a uniform
positive lower bound for the right-hand side.
\begin{corollary}\label{cor:main}
Let $(M,\omega)$ be a connected compact Hermitian manifold, and let
$k\in\{1,\ldots,n\}$. Suppose that $\chi$ is a smooth $(1,1)$-form with
$\lambda(g^{-1}\chi)\in\Gamma_k$ pointwise on $M$. Let $u$ be a bounded
pluripotential solution of
\begin{equation}\label{eq:main-cor}
(\chi+dd^c u)^k\wedge\omega^{n-k}=g\,\omega^n,
\end{equation}
where $g\in L^{p_0}$ is nonnegative and $p_0>\frac{n}{k}$. Then
$u\in C^{\mu}(M)$ for every
\[
0<\mu< \frac{2(kp_0-n)}{kp_0(n+1)-n}.
\]
\end{corollary}

\subsection{Envelope}

Define the viscosity envelope:
\[
 P_{\Gamma,\chi}(h)
 =\left(\sup\left\{v\in\operatorname{USC}(M):
       v\leq h,\quad
       \lambda[\chi+\ddc v]\in\overline\Gamma
       \text{ in the viscosity sense}\right\}\right)^*.
\]
Here $(\cdot)^*$ denotes upper semicontinuous regularization.
The strict admissibility of $\chi$ makes this family nonempty.
General qualitative envelope theory fits
the subequation framework of Harvey--Lawson \cite{HarveyLawson2011}; for the
Monge--Amp\`ere and complex Hessian cones, related rooftop and obstacle
envelopes have been studied in
\cite{Berman2019,GuedjLuZeriahi2019,DarvasRubinstein2016,ChuMcCleerey2021,AhagCzyzLuRashkovskii2024}.

We use the following penalized equation to approximate the envelope:
\begin{equation}\label{eq:intro-penalized}
 f(\lambda[\chi+\ddc u_{\beta,h}])
 =e^{\beta(u_{\beta,h}-h)},
 \qquad \lambda[\chi+\ddc u_{\beta,h}]\in\Gamma,
\end{equation}

\begin{theorem}\label{thm:penalized-c2}
Suppose that $(f,\Gamma)$ satisfies Assumptions~\ref{ass1}
and~\ref{ass2}, and that $\lambda(g^{-1}\chi)\in\Gamma$ on $M$.
For $h\in C^2(M)$, there are $\beta_0>0$ and constants $C_1,C_2$
such that every smooth admissible solution of (\ref{eq:intro-penalized})
with $\beta\geq\beta_0$ satisfies
\begin{equation}\label{eq:penalized-c11}
\begin{split}
\frac{-C_1}{\beta} + \min_M h \le \min_M u_{\beta,h} & \le \max_M u_{\beta,h} \le \max_M h + \frac{C_1}{\beta}\\
\|\partial\bar\partial u_{\beta,h}\|_{\infty}
 +\|\nabla u_{\beta,h}\|_{\infty}^2
 &\le C_2\bigl(1+\|\nabla^2h\|_{\infty}
                    +\|\nabla h\|_{\infty}^2\bigr).
\end{split}
\end{equation}
The constant $C_1$ depends on the fixed data $(f,\omega,\chi)$, while
$C_2$ may additionally depend on $\osc_M h$, but neither depends on
$\beta$ or on the derivative norms of $h$.
\end{theorem}

\begin{remark}
The proof does not invoke the level-set $\mathcal{C}$-subsolution dichotomy of
\cite{S}; instead it uses the fixed cone-interiority margin of $\chi$ and the good term $e^{\beta u_{\beta,h}}$.
The explicit dependence on the derivatives of $h$ is needed for
Theorem~\ref{thm:main}.
\end{remark}

\begin{corollary}\label{cor:penalized-c1alpha}
Suppose that $(f,\Gamma)$ satisfies Assumptions~\ref{ass1}--\ref{ass3}
and that $\lambda(g^{-1}\chi)\in\Gamma$.  If $h\in C^\infty(M)$, the
unique smooth solutions supplied by
\cite[Theorem~3.1(1)]{ChengXu2025} satisfy, for every $0<\alpha<1$,
\[
 \lim_{\beta\to\infty}
 \|u_{\beta,h}-P_{\Gamma,\chi}(h)\|_{C^{1,\alpha}(M)}=0.
\]
\end{corollary}

Theorem~\ref{thm:main} does not assume Assumption~\ref{ass2}. To apply
Theorem~\ref{thm:penalized-c2} in the proof of Theorem~\ref{thm:main}, we
therefore construct an auxiliary operator.

\begin{theorem}\label{thm:canonical-operator}
For every cone $\Gamma$ satisfying Assumption~\ref{ass1}(1), there is a
function $f_\Gamma$ such that $(f_\Gamma,\Gamma)$ satisfies
Assumptions~\ref{ass1}--\ref{ass3}.
\end{theorem}

\begin{corollary}\label{cor:envelope-regularity}
Let $\Gamma$ satisfy Assumption~\ref{ass1}(1), assume
$\lambda(g^{-1}\chi)\in\Gamma$, and let $h\in C^2(M)$.  Then, for every
$0<\alpha<1$ and every finite $p$,
\[
 P_{\Gamma,\chi}(h)\in C^{1,\alpha}(M)\cap W^{2,p}(M),
 \qquad
 \ddc P_{\Gamma,\chi}(h)\in L^\infty(M).
\]
Moreover, its Lipschitz and complex-Hessian bounds have the quantitative
dependence displayed in \eqref{eq:penalized-c11}.
\end{corollary}

\begin{remark}
The conclusion $\ddc P\in L^\infty$ is not, by itself, the standard real
$C^{1,1}$ estimate.  For the complex Hessian cones $\Gamma_m$, the stronger
optimal real $C^{1,1}$ regularity is proved by Chu--McCleerey
\cite{ChuMcCleerey2021} using an additional real-Hessian argument.
\end{remark}

\subsection{Regularization strategy}

The main point is a regularization mechanism that is independent of the
equation satisfied by $u$.
Given a function $u\in SH_{\chi,\Gamma}(M,\omega)\cap L^{\infty}(M)$, let
\[
 H_t=\exp(t\Delta_\omega^{\mathrm{Ch}}),
 \qquad C_\chi=\sup_M\operatorname{tr}_\omega\chi,
\]
where $\Delta_{\omega}^{\mathrm{Ch}}$ is the Chern Laplacian of the Hermitian
metric $\omega$. We define
\begin{equation}\label{eq:intro-regularization}
 \mathcal R_tu
 :=P_{\Gamma,\chi}(h_t),
 \qquad h_t:=H_tu+C_\chi t.
\end{equation}

Let $f_{\Gamma}$ be the function defined in
Theorem~\ref{thm:canonical-operator}. We solve
\begin{equation}
 f_{\Gamma}(\lambda[\chi+\ddc u_{\beta,h_t}])
 =e^{\beta(u_{\beta,h_t}-h_t)},
 \qquad \lambda[\chi+\ddc u_{\beta,h_t}]\in\Gamma.
\end{equation}
We then choose $\beta(t)$ sufficiently large that
$\|u_{\beta(t),h_t}-\mathcal R_tu\|_{\infty}\leq t$.

We use $u_t:=u_{\beta(t),h_t}$ as our regularization of $u$. It satisfies the
following properties.
\begin{theorem}\label{thm:regularization}
For every $u\in SH_{\chi,\Gamma}(M,\omega)\cap L^{\infty}(M)$, the
regularization $u_t$, $0<t\leq1$, satisfies:
\begin{enumerate}
\item $u_t\in C^{\infty}(M)\cap SH_{\chi,\Gamma}(M,\omega)$.
\item $\|u_t-u\|_{L^1}\leq Ct$, where $C$ depends on $\|u\|_{\infty}$.
\item $u_t\geq u-t$.
\item $\Lip(u_t)\leq Ct^{-1/2}$, where $C$ depends on $\|u\|_{\infty}$.
\item $||dd^c u_t||_{\infty}\le C t^{-1}$, where $C$ depends on $\|u\|_{\infty}$.
\item $\|u_t\|_{\infty}\leq C$, where $C$ depends on
      $\|u\|_{\infty}$.
\end{enumerate}
\end{theorem}

Although the theorem assumes that $u\in L^{\infty}$, the regularization $u_t$
can also be defined for unbounded $u$, provided that the Chern heat semigroup
$H_t$ acts on $u$. This observation may be useful when $\Gamma=\Gamma_k$,
because quasi-$m$-subharmonic functions in $SH_{\chi,\Gamma_k}$ are well
studied and may be unbounded while satisfying suitable integrability
conditions; see \cite{Fang2026}. It would be interesting to study the behavior
of $u_t$ for $u\in SH_{\chi,\Gamma_k}$, in analogy with Demailly's
regularization of quasi-plurisubharmonic functions \cite{Demailly1992}.

The paper is organized as follows. Section~3 collects the viscosity,
pluripotential, and cone-theoretic preliminaries. Section~4 develops the
Chern heat-semigroup estimates used in the regularization. Section~5 studies
the penalized obstacle equation and proves estimates uniform in the
penalization parameter. Section~6 constructs the canonical cone operator and
establishes regularity of the associated envelope. Section~7 proves
Theorem~\ref{thm:regularization}. Section~8 establishes the quantitative
stability estimate, and Section~9 combines regularization and stability to
prove Theorem~\ref{thm:main} and Corollary~\ref{cor:main}.

\section{Acknowledgments}
The author thanks Professor Jingrui Cheng for introducing him to this problem and for
their collaboration when the author was a second-year PhD student. The author also
thanks Professor Duong H. Phong for his encouragement and helpful comments on an
earlier version of this paper.
\section{Preliminary}

\begin{defn}
\[
\Gamma_k=\{\lambda \in \R^n: \sigma_i(\lambda)>0 \text{ for } 1\leq i\leq k\}.
\]
Here
$\sigma_i(\lambda)=\sum_{1\leq j_1<j_2<\cdots<j_i\leq n}
\lambda_{j_1}\lambda_{j_2}\cdots\lambda_{j_i}$.
\end{defn}

We shall repeatedly use the spectral matrix cone
\[
 \mathcal C_\Gamma
 :=\{A\in\operatorname{Herm}(n):\lambda(A)\in\Gamma\}.
\]
Since $\Gamma$ is symmetric and convex, $\mathcal C_\Gamma$ is convex.
Moreover,
\begin{equation}
 \mathcal C_\Gamma+\overline{\mathcal C_{\Gamma_n}}
 \subset\mathcal C_\Gamma.
 \label{eq:spectral-cone-monotonicity}
\end{equation}
Indeed, the convexity assertion is the standard spectral-convexity theorem;
the monotonicity follows from Weyl monotonicity and
$\Gamma+\overline{\Gamma_n}\subset\Gamma$.

\subsection{Viscosity theory}

We first explain the meaning of ``touching from above or below.''
\begin{defn}\label{touch}
Let $\varphi$ be a function defined on $M$ and $x_0\in M$.  Let $\psi$ be another function defined on an open subset of $M$ containing $x_0$.  
\begin{enumerate}
\item We say that $\psi$ touches $\varphi$ from above at $x_0$ if there exists an open neighborhood $U$ of $x_0$ such that $\psi(x_0)=\varphi(x_0)$ and $\psi\geq\varphi$ on $U$.
\item We say that $\psi$ touches $\varphi$ from below at $x_0$ if there exists an open neighborhood $U$ of $x_0$ such that $\psi(x_0)=\varphi(x_0)$ and $\psi\leq\varphi$ on $U$.
\end{enumerate}
\end{defn}

\begin{defn}
Let $G\in C(M)$, $c\in\R$, and $\varphi\in C(M)$. We say that $\varphi$
is a viscosity solution of \eqref{eq:main} if the following conditions hold:
\begin{enumerate}
\item For any $x_0\in M$ and any $C^2$ function $P$ defined in a neighborhood of $x_0$ that touches $\varphi$ from above at $x_0$,  one has
\begin{equation*}
\lambda[\chi+dd^cP](x_0)\in \overline{\Gamma},\,\,f\big(\lambda[\chi+dd^cP]\big)(x_0)\ge e^{G(x_0)+c}.
\end{equation*}
\item For any $x_0\in M$ and any $C^2$ function $P$ defined in a neighborhood of $x_0$ that touches $\varphi$ from below at $x_0$, either $\lambda[\chi+dd^cP](x_0)\notin \Gamma$, or $\lambda[\chi+dd^cP](x_0)\in \Gamma$ and $f\big(\lambda[\chi+dd^cP]\big)(x_0)\leq e^{G(x_0)+c}$.
\end{enumerate}
\end{defn}

\begin{defn}\label{d2.2}
We say that $\lambda[\chi+ dd^c \varphi ]\in \overline{\Gamma}$ in the viscosity sense, if for any $x_0\in M$ and any $C^2$ function $P$ defined in a neighborhood of $x_0$ that touches $\varphi$ from above at $x_0$, one has:
\begin{equation*}
\lambda[\chi+ dd^c P](x_0)\in \overline{\Gamma}.
\end{equation*}
We define $SH_{\chi,\Gamma}(M,\omega)=\{u \in USC(M): \lambda[\chi+ dd^c u] \in \overline{\Gamma} \text{ in the viscosity sense}\}$.
\end{defn}

\subsection{Pluripotential theory}

\begin{defn}
Let $\Omega\subset\C^n$ be an open domain equipped with a Hermitian metric
$\omega$. An upper semicontinuous function
\[
u\colon \Omega\longrightarrow[-\infty,+\infty)
\]
is called $m$-$\omega$-subharmonic if
$u\in L_{\mathrm{loc}}^1(\Omega)$ and, for any collection
\[
\gamma_1,\ldots,\gamma_{m-1}\in\Gamma_m(\omega),
\]
we have
\[
dd^c u\wedge\gamma_1\wedge\cdots\wedge\gamma_{m-1}
\wedge\omega^{n-m}\geq 0
\]
in the sense of currents.
\end{defn}

\begin{defn}
\cite[Definition 2.4 and Lemma 9.10]{GuNguyen2018} Let $(M,\omega)$ be a Hermitian manifold. Let $\chi$ be a smooth $(1,1)$ form. A function
\[
u\colon M\longrightarrow[-\infty,+\infty)
\]
is called $(\chi,m)$-$\omega$-subharmonic if it can be written
locally as a sum of a smooth function and an $\omega$-sh function, and
globally, for any collection
\[
\gamma_1,\ldots,\gamma_{m-1}\in\Gamma_m(M,\omega),
\]
we have
\begin{equation}
(\chi+dd^c u)\wedge\gamma_1\wedge\cdots\wedge\gamma_{m-1}
\wedge\omega^{n-m}\geq 0
\qquad\text{on }M
\end{equation}
in the weak sense of currents. Denote by
$SH_{\chi,m}(M,\omega)$, or simply $SH_{\chi,m}(\omega)$, the set
of all $(\chi,m)$-$\omega$-sh functions on $M$.
\end{defn}

\begin{defn}\cite[Section 9]{KNmeasure}
If $\Omega$ is a local coordinate chart on $M$ and $\rho$ is a
strictly psh function on $\Omega$ such that
\[
dd^c\rho\geq\chi
\qquad\text{on }\Omega,
\]
then $u+\rho$ is an $m$-$\omega$-sh function on $\Omega$. This observation
extends the definition of the wedge product for
currents associated with bounded $(\chi,m)$-$\omega$-sh functions
by using a partition of unity.
Namely, write
\[
\tau=\chi-dd^c\rho,
\]
which is a smooth $(1,1)$-form. Then
\[
\chi+dd^c u=dd^c(u+\rho)+\tau.
\]

We define
\begin{align*}
H_{\chi,m}(u)
:=(\chi+dd^c u)^m\wedge\omega^{n-m}
&:=\sum_{j=0}^{m}\binom{m}{j}
   \bigl[dd^c(u+\rho)\bigr]^j
   \wedge\tau^{m-j}\wedge\omega^{n-m} \\
&=\sum_{j=0}^{m}\binom{m}{j}
   \mathcal{L}_j(u+\rho)\wedge\tau^{m-j}.
\end{align*}
\end{defn}

\subsection{\texorpdfstring{$\Gamma$}{Gamma}-solutions}
\begin{defn}\label{def:gamma-solution}
Following \cite[Definition~15]{S}, suppose
$u\colon\mathbb{C}^n\to\mathbb{R}$ is continuous. We say that
$u$ is a (viscosity) $\Gamma$-subsolution if, for all $h\in C^2$ such
that $u-h$ has a local maximum at $z$, we have
\[
\lambda\bigl(h_{i\bar{j}}\bigr)\in\overline{\Gamma},
\]
where $\lambda(A)$ denotes the eigenvalues of the Hermitian matrix $A$.

We say that $u$ is a $\Gamma$-solution if it is a
$\Gamma$-subsolution and, in addition, for all $z\in\mathbb{C}^n$, if
$h\in C^2$ and $u-h$ has a local minimum at $z$, then
\[
\lambda\bigl(h_{i\bar{j}}(z)\bigr)
\in\mathbb{R}^n\setminus\Gamma.
\]
\end{defn}

\begin{lemma}\label{lem:gamma-liouville}
\cite[Theorem~20]{S}.  Let $u\colon \mathbb{C}^n\to\mathbb{R}$ be a Lipschitz
$\Gamma$-solution such that $|u|<C$ and $u$ has Lipschitz constant
bounded by $C$. Then $u$ is constant.
\end{lemma}

\subsection{Property of f}

\begin{lemma}\label{lem:concavity-pairing}
Suppose that $(f,\Gamma)$ satisfies Assumption~\ref{ass1}. Then
$\sum_i \lambda_i f_i(\lambda)\leq f(\lambda)$ for every $\lambda\in\Gamma$.
\end{lemma}
\begin{proof}
By concavity of $f$,
$\sum_i \lambda_i f_i(\lambda)\leq f(\lambda)-f(0)\leq f(\lambda)$.
\end{proof}

\begin{lemma}[Determinant domination]\label{lem:determinant-domination}
Under Assumptions~\ref{ass1} and~\ref{ass3}, for every
$\lambda\in\Gamma_n$,
\begin{equation}
 f(\lambda)\geq n\gamma_0^{1/n}
 (\lambda_1\cdots\lambda_n)^{1/n}.
 \label{eq:determinant-domination}
\end{equation}
\end{lemma}
\begin{proof}
At a diagonal matrix with eigenvalues $\lambda$, the eigenvalues of the
linearization are $f_i(\lambda)$.  Hence Assumption~\ref{ass3} gives
$\prod_i f_i(\lambda)\geq\gamma_0$.  The AM-GM inequality
then implies
\begin{align*}
 f(\lambda)
 &\geq \sum_i\lambda_i f_i(\lambda)+ f(0) \\
 & =  \sum_i\lambda_i f_i(\lambda) \\
 &\geq n\left(\prod_i\lambda_i f_i(\lambda)\right)^{1/n} \\
 &\geq n\gamma_0^{1/n}
 (\lambda_1\cdots\lambda_n)^{1/n},
\end{align*}
which is \eqref{eq:determinant-domination}.  
\end{proof}

\begin{lemma}\label{lem:rhs-growth}
Suppose that $(f,\Gamma)$ satisfies Assumptions~\ref{ass1}
and~\ref{ass2}.  There is a fixed constant $C$ such that every
$h\in C^2(M)$ with $\lambda[\chi+\ddc h]\in\Gamma$ satisfies
\begin{equation}
 f(\lambda[\chi+\ddc h])
 \leq C\bigl(1+\|\partial\bar\partial h\|_\infty\bigr).
 \label{eq:rhs-growth}
\end{equation}
\end{lemma}
\begin{proof}
For a fixed $C_0$ depending only on $\omega$ and $\chi$, the Hermitian
matrix of $\chi+\ddc h$ is bounded above by
$T\omega$, where
$T=C_0(1+\|\partial\bar\partial h\|_\infty)$.  By ellipticity and
concavity,
\[
 f(\lambda[\chi+\ddc h])\leq f(T\one)= f(\one) + \sum_i f_i(\one)(T-1) \le C_1 T,
\]
which proves the claim.
\end{proof}

\begin{lemma}\label{lem:chengxu-pairing}  \cite[Lemma~4.7]{ChengXu2025}
Let $A,B$ be Hermitian matrices whose spectra belong to $\Gamma$.  Then
\begin{equation}
 \operatorname{tr}\bigl(DF(A)B\bigr)\geq0.
 \label{eq:matrix-pairing}
\end{equation}
\end{lemma}

\begin{lemma}\label{lem:cone-translation}
Suppose that $(f,\Gamma)$ satisfies Assumptions~\ref{ass1}
and~\ref{ass2}.  For every $\sigma>0$ there is $N>0$ such that
\begin{equation}
 \Gamma+N\one\subset\Gamma^\sigma
 :=\{\lambda\in\Gamma:f(\lambda)>\sigma\}.
 \label{eq:cone-translation}
\end{equation}
\end{lemma}
\begin{proof}
Choose $N$ so that $Nf(\one)>\sigma$.  Concavity and degree-one
homogeneity imply superadditivity on $\Gamma$:
$f(x+y)\geq f(x)+f(y)$.  Hence, for $\lambda\in\Gamma$,
$f(\lambda+N\one)\geq f(\lambda)+Nf(\one)>\sigma$.
\end{proof}

\begin{defn}
Suppose, as in the introduction, that $(M,\alpha)$ is Hermitian and
$\chi$ is a real $(1,1)$-form. We say that $\underline{u}$ is a
$\mathcal{C}$-subsolution for the equation
\[
F(A)=h
\]
if, at each $x\in M$, the set
\begin{equation}
\left(
\lambda\left[
\alpha^{j\bar{p}}
\bigl(\chi_{i\bar{p}}+\underline{u}_{i\bar{p}}\bigr)
\right]
+\Gamma_n
\right)
\cap\partial\Gamma^{h(x)}
\end{equation}
is bounded. Let us also say that $\underline{u}$ is admissible if
\[
\lambda\left[
\alpha^{j\bar{p}}
\bigl(\chi_{i\bar{p}}+\underline{u}_{i\bar{p}}\bigr)
\right]\in\Gamma.
\]
\end{defn}

\section{Chern heat flow}

The first regularization step uses the Chern heat flow.  Consider
the complex, or Chern, Laplacian
\[
 \Delta^{\mathrm{Ch}}_\omega q
 =\operatorname{tr}_\omega(\ddbar q)
 =g^{j\bar k}q_{j\bar k}.
\]
Let
\[
 H_t=e^{t\Delta_{\omega}^{\mathrm{Ch}}},
 \qquad
 H_tu(x)=\int_MK_t(x,y)u(y)\,dV_{g_{\mathbb R}}(y).
\]
This is the Chern heat flow, characterized by
\begin{equation}\label{eq:chern-diffusion}
 \partial_tH_tq=\Delta^{\mathrm{Ch}}_\omega H_tq,
 \qquad H_t q |_{t=0}=q.
\end{equation}

The maximum principle gives
\begin{equation}\label{eq:markov}
 K_t(x,y)\geq0,
 \qquad
 \int_MK_t(x,y)\,dV_{g_{\mathbb R}}(y)=1.
\end{equation}

\begin{lemma}\label{prop:kernel}
Let $(M,\omega)$ be a compact Hermitian manifold of complex dimension
$n$.  For $k=0,1,2$ there are constants $c>0$ and $C_k<\infty$, depending only
on the fixed smooth Hermitian background, such that
\begin{equation}\label{eq:gaussian-derivatives1}
 |\nabla_x^kK_t(x,y)|
 \leq C_k t^{-(2n+k)/2}
 \exp\!\left(-\frac{d_{g_{\mathbb R}}(x,y)^2}{ct}\right),
 \qquad 0<t\leq1.
\end{equation}
\end{lemma}

\begin{proof}
Equation \eqref{eq:chern-diffusion} lies in the class of uniformly parabolic
operators with bounded H\"older coefficients and bounded lower-order
coefficients.  The general fundamental-solution theorem of
Porper--Eidel'man \cite[Theorems~1.1--1.2]{PorperEidelman1984} gives, in
Euclidean coordinates,
\[
 |\partial_x^k Z(t,x;0,y)|
 \leq C_k t^{-(2n+k)/2}
 \exp\!\left(-\frac{|x-y|^2}{ct}\right),
 \qquad k\leq2.
\]
Their model equation explicitly allows first-order and zeroth-order terms.
Applying the local
parametrix construction in a finite atlas, and using compactness to make the
ellipticity and coefficient bounds uniform, gives the estimate near the
diagonal on $M$.  The semigroup identity supplies the off-diagonal estimate.
Coordinate derivatives and covariant derivatives of order at most two
differ only by uniformly bounded lower-order terms on this finite atlas,
which proves \eqref{eq:gaussian-derivatives1}.
\end{proof}

Integrating \eqref{eq:gaussian-derivatives1} in geodesic polar coordinates
gives
\begin{lemma}
Let $(M,\omega)$ be a compact Hermitian manifold.  There is a constant
$C$, depending only on the fixed smooth Hermitian background, such that
\begin{equation}\label{eq:integrated-kernel}
 \sup_{x\in M}\int_M|\nabla_xK_t(x,y)|\,dV_{g_{\mathbb R}}(y)
 \leq Ct^{-1/2},
 \qquad
 \sup_{x\in M}\int_M|\nabla_x^2K_t(x,y)|\,dV_{g_{\mathbb R}}(y)
 \leq Ct^{-1},\,\,\, 0<t\le 1.
\end{equation}
\end{lemma}

\begin{lemma}\label{prop:smoothing}
\begin{enumerate}
\item For every $u\in L^\infty(M)$ and $0<t\leq1$,
\begin{equation}\label{eq:smoothing}
 \|\nabla H_tu\|_{L^\infty}
 \leq Ct^{-1/2}\osc_M u,
 \qquad
 \|\nabla^2H_tu\|_{L^\infty}
 \leq Ct^{-1}\osc_M u.
\end{equation}
Consequently,
\begin{equation}\label{eq:complex-hessian}
 \|\partial\bar\partial H_tu\|_{L^\infty}
 \leq Ct^{-1}\osc_M u.
\end{equation}
\item For every $u\in C^{\alpha}(M)$ and $0<t\leq1$, the following
estimates hold:
\[
 \|\nabla H_tu\|_{\infty}
 \leq Ct^{-\frac{1-\alpha}{2}}[u]_{C^{\alpha}},
 \qquad
 \|\nabla^2H_tu\|_{\infty}
 \leq Ct^{-\frac{2-\alpha}{2}}[u]_{C^{\alpha}}.
\]
\end{enumerate}
\end{lemma}

\begin{proof}
Since $\Delta_\omega^{\mathrm{Ch}}1=0$, equation \eqref{eq:markov} implies
\begin{equation}\label{e zero intk}
 \int_M\nabla_x^kK_t(x,y)\,dV_{g_{\mathbb R}}(y)=0,
 \qquad k=1,2.
\end{equation}
First, consider the case $u\in L^{\infty}(M)$.
Choose $c=(\sup_M u+\inf_M u)/2$.  Then
\[
 \|u-c\|_\infty\leq\frac12\osc_M u.
\]
Differentiating the kernel representation and applying
\eqref{eq:integrated-kernel} proves both inequalities in
\eqref{eq:smoothing}.  On a fixed Hermitian manifold the complex Hessian is
bounded by the real covariant Hessian plus a fixed multiple of the gradient.
For $0<t\leq1$, the latter $t^{-1/2}$ term is absorbed by $t^{-1}$, proving
\eqref{eq:complex-hessian}.

Second, suppose that $u\in C^{\alpha}(M)$. By
\eqref{e zero intk}, for $j=1,2$,
\[
 \nabla^jH_tu(x)
 =\int_M\nabla_x^jK_t(x,y)\bigl(u(y)-u(x)\bigr)
   \,dV_{g_{\mathbb R}}(y).
\]
Using
$|u(y)-u(x)|\leq [u]_{C^{\alpha}}d(x,y)^{\alpha}$, we obtain
\[
\begin{split}
 |\nabla^jH_tu(x)|
 &\leq C[u]_{C^{\alpha}}t^{-(2n+j)/2}
 \int_M e^{-\frac{d^2(x,y)}{Ct}}d(x,y)^{\alpha}
 \,dV_{g_{\mathbb R}}(y) \\
 &\leq C[u]_{C^{\alpha}}
 t^{-\frac{2n+j}{2}}t^{\frac{2n+\alpha}{2}} \\
 &=C[u]_{C^{\alpha}}t^{\frac{\alpha-j}{2}}.
\end{split}
\]
\end{proof}

\begin{lemma}
\label{lem:chern-gradient-interpolation}
There is a constant $C$ depending only on the fixed Hermitian background
such that every smooth real-valued function $w$ satisfies
\begin{equation}\label{eq:chern-gradient-interpolation}
 \|\nabla w\|_\infty
 \leq C\left(\|w\|_\infty+
 \|w\|_\infty^{1/2}
 \|\Delta_\omega^{\mathrm{Ch}}w\|_\infty^{1/2}\right).
\end{equation}
\end{lemma}

\begin{proof}
Write $A=\|w\|_\infty$ and
$D=\|\Delta_\omega^{\mathrm{Ch}}w\|_\infty$.
The semigroup identity gives, for $0<s\leq1$,
\[
 w=H_sw-\int_0^sH_\tau\Delta_\omega^{\mathrm{Ch}}w\,d\tau.
\]
By \eqref{eq:integrated-kernel},
\[
 \|\nabla w\|_\infty
 \leq C\left(s^{-1/2}A+s^{1/2}D\right).
\]
The derivative of the integral is justified by the integrable bound
$C\tau^{-1/2}D$. If $A=0$, then $w=0$. If $D=0$, or if
$D\leq A$, take $s=1$. Otherwise take $s=A/D\in(0,1)$.
These choices prove \eqref{eq:chern-gradient-interpolation}.
Self-adjointness of the Chern heat semigroup is not needed.
\end{proof}

\begin{lemma}
\label{lem:dyadic-holder}
Let $u\in C(M)$ and let $z_t\in C^\infty(M)$, $0<t\leq1$.
Suppose that, for some $0<\gamma<b$ and $A_0,D_0>0$,
\begin{equation}\label{e dya holder zu}
 \|z_t-u\|_\infty\leq A_0t^\gamma,
 \qquad
 \|\Delta_\omega^{\mathrm{Ch}}z_t\|_\infty\leq D_0t^{-b}.
\end{equation}
Then
\begin{equation}\label{eq:dyadic-holder-exponent}
 u\in C^{\frac{2\gamma}{\gamma+b}}(M).
\end{equation}
The norm is bounded in terms of the fixed geometry, $\gamma,b,A_0,D_0$
and $\|u\|_\infty$.
\end{lemma}

\begin{proof}
Set $t_j=2^{-j}$, $z_j=z_{t_j}$ and $w_j=z_{j+1}-z_j$.
The assumptions give
\[
 \|w_j\|_\infty\leq Ct_j^\gamma,
 \qquad
 \|\Delta_\omega^{\mathrm{Ch}}w_j\|_\infty\leq Ct_j^{-b}.
\]
Lemma~\ref{lem:chern-gradient-interpolation} therefore yields
\begin{equation}\label{eq:dyadic-gradient-bound}
 \|\nabla w_j\|_\infty
 \leq Ct_j^{(\gamma-b)/2}.
\end{equation}
Here $t_j^\gamma$ is absorbed because $0<t_j\leq1$ and $\gamma<b$.
The same interpolation lemma bounds $\|\nabla z_0\|_\infty$ by the
stated data. Uniform convergence gives the identity
$u=z_0+\sum_{j\geq0}w_j$.

Put $\sigma=2\gamma/(\gamma+b)\in(0,1)$.
For $0<r=d(x,y)<1$, choose $J\geq0$ with
\begin{equation}\label{e tj r}
 t_{J+1}^{(\gamma+b)/2}<r\leq t_J^{(\gamma+b)/2}.
\end{equation}
Use \eqref{eq:dyadic-gradient-bound} for the terms $j<J$ and the
supremum bound for $j\geq J$. The two geometric sums give
\begin{align*}
 |u(x)-u(y)| &= | \big(z_0(x)+\sum_{j\geq0}w_j(x) \big) - \big( z_0(y)+\sum_{j\geq0}w_j (y) \big)| \\
 &\leq |z_0(x)- z_0(y)| + \sum_{j =0}^{J-1} ( |w_j(x)|+ |w_j(y)| )+ 
 \sum_{j=J}^{\infty} ( |w_j(x)|+ |w_j(y)| ) \\
 & \le C r + 2r \sum_{j=0}^{J-1} ||\nabla w_j||_{\infty} + 2 \sum_{j=J}^{\infty} ||w_j||_{\infty}\\
 & \le Cr + Cr \sum_{j=0}^{J-1} t_j^{\frac{\gamma-b}{2}} + C \sum_{j=J}^{\infty} t_j^{\gamma} \\
 & \le C(r+ r t_J^{\frac{\gamma-b}{2}} + t_J^{\gamma}).
 \end{align*}
 The constant $C$ in the third line of the above calculation depends on $||\nabla z_0||_{\infty}$, which according to Lemma \ref{lem:chern-gradient-interpolation} depends on $A_0, D_0$ and $||u||_{\infty}$. In the fourth line of the above calculation, we use (\ref{e dya holder zu}) and (\ref{eq:dyadic-gradient-bound}). Plug (\ref{e tj r}) into the above formula, we can get:
 \begin{equation*}
 |u(x)-u(y)|\le C r^{\frac{2\gamma}{\gamma+b}}.
 \end{equation*} 

\end{proof}

\begin{proposition}
\label{prop:heat-envelope}
Let $u\in L^{\infty}(M)$ be
$(\Gamma,\chi)$-admissible.  Put
\[
 C_\chi:=\sup_M\operatorname{tr}_\omega\chi,
\]
and, for $t>0$, define
\begin{equation}\label{eq:heat-envelope}
 h_t:=H_tu+C_\chi t.
\end{equation}
Then
\begin{equation}\label{eq:sandwich}
 u \leq h_t \le \max_M u + C_{\chi} t,
\end{equation}
and
\begin{equation}\label{eq:L1-rate}
\|h_t-u\|_{L^1(dV_\omega)}
 \leq C_{\omega,\chi}t.
\end{equation}
Moreover,
\begin{equation}\label{eq:smoothing1}
 \|\nabla h_t\|_{L^\infty}
 \leq Ct^{-1/2}\osc_M u,
 \qquad
 \|\nabla^2h_t\|_{L^\infty}
 \leq Ct^{-1}\osc_M u.
\end{equation}
\end{proposition}

\begin{proof}
Since $\Gamma\subset\Gamma_1$, every  $\phi$ touching $u$ from above  satisfies the following inequality at the touching set:
\[
 0\leq\sum_{i=1}^n\lambda_i(A_\phi)
 =\operatorname{tr}_\omega\chi+\Delta^{\mathrm{Ch}}_\omega\phi
 \leq C_\chi+\Delta^{\mathrm{Ch}}_\omega\phi.
\]
Thus
\begin{equation}\label{eq:trace-measure}
 \mu:=C_\chi+\Delta^{\mathrm{Ch}}_\omega u\geq0
\end{equation}
in the viscosity sense.  For this linear inequality, viscosity and
distributional subsolutions agree.  Hence $\mu$ is a nonnegative
distribution and therefore a nonnegative Radon measure.

First suppose that $u$ is smooth.  The semigroup identities give
\[
 \frac{d}{ds}H_su
 =\Delta^{\mathrm{Ch}}_\omega H_su
 =H_s\Delta^{\mathrm{Ch}}_\omega u,
 \qquad H_0u=u.
\]
Moreover, $\Delta^{\mathrm{Ch}}_\omega1=0$, so the Markov property gives
$H_sC_\chi=C_\chi$.  The fundamental theorem of calculus then yields
\begin{align}
 h_t-u
 &=H_tu-u+C_\chi t \notag\\
 &=\int_0^tH_s(\Delta^{\mathrm{Ch}}_\omega u)\,ds
   +\int_0^tH_s(C_\chi)\,ds \notag\\
 &=\int_0^tH_s(C_\chi+\Delta^{\mathrm{Ch}}_\omega u)\,ds \notag\\
 &=\int_0^tH_s\mu\,ds.                                  \label{eq:semigroup-identity}
\end{align}

For merely continuous admissible $u$, the same formula is understood as an
improper distributional identity.  The linear viscosity inequality in
\eqref{eq:trace-measure} is also distributional, and a positive distribution
is a positive Radon measure.  For every $0<\varepsilon<t$, semigroup
smoothing gives
\begin{equation}\label{eq:epsilon-semigroup}
 H_tu-H_\varepsilon u+C_\chi(t-\varepsilon)
 =\int_\varepsilon^tH_s\mu\,ds.
\end{equation}
The Chern heat kernel is positive, hence $H_s\mu\geq0$.  Therefore the
right-hand side of \eqref{eq:epsilon-semigroup} is nonnegative.  Since $u$
is continuous on the compact manifold, $H_\varepsilon u\to u$ uniformly as
$\varepsilon\downarrow0$.  Passing to the limit proves
\[
 h_t-u=\int_0^tH_s\mu\,ds\geq0,
\]
where the integral from zero denotes the preceding improper limit. By maximum principle, we have that $H_t u \le \max_M u$, which implies that:
\begin{equation*}
h_t \le \max_M u + C_{\chi}t.
\end{equation*}
Thus, we conclude the proof of (\ref{eq:sandwich}).

Finally, let $\rho_\omega>0$ be the normalized Gauduchon density
\cite{Gauduchon1977}, so that
\begin{equation}\label{eq:gauduchon-density}
 \ddbar(\rho_\omega\omega^{n-1})=0,
 \qquad
 \int_M\rho_\omega\,dV_\omega=1.
\end{equation}
Integration by parts gives, for every smooth $q$,
\[
 \int_M\Delta^{\mathrm{Ch}}_\omega q\,
       \rho_\omega\,dV_\omega=0.
\]
Consequently the possibly non-self-adjoint Chern heat semigroup preserves
the weighted measure:
\begin{equation}\label{eq:weighted-mass}
 \int_MH_tq\,\rho_\omega\,dV_\omega
 =\int_Mq\,\rho_\omega\,dV_\omega.
\end{equation}
Indeed, the derivative in $t$ of the left-hand side vanishes by the
preceding integration-by-parts identity.

Since the positive smooth function $\rho_\omega$ has a positive minimum,
\eqref{eq:weighted-mass} yields
\begin{align*}
 \|h_t-u\|_{L^1(dV_\omega)}
 &=\int_M(h_t-u)\,dV_\omega\\
 &\leq(\min_M\rho_\omega)^{-1}
       \int_M(h_t-u)\rho_\omega\,dV_\omega\\
 &=C_\chi(\min_M\rho_\omega)^{-1}t.
\end{align*}
Here the first equality uses \eqref{eq:sandwich}.  
Finally, \eqref{eq:smoothing1} follows from \eqref{eq:smoothing}.
\end{proof}

\begin{remark}
The function $h_t$ is smooth for every $t>0$, but it need not be
$(\Gamma,\chi)$-admissible.  The heat-flow argument uses only the scalar
inequality
$\operatorname{tr}_\omega\chi+\Delta^{\mathrm{Ch}}_\omega u\geq0$.
This
is why the second step uses an envelope.
\end{remark}

\section{The penalized obstacle equation}
In this section, we assume that $(f,\Gamma)$ satisfies
Assumptions~\ref{ass1} and~\ref{ass2}.
For \(h\in C^\infty(M)\) and large \(\beta\), consider
\begin{equation} \label{eq:penalized}
 \log f\bigl(\lambda[\chi+\ddc u_{\beta,h}]\bigr)
 =\beta(u_{\beta,h}-h),
 \qquad
 \lambda[\chi+\ddc u_{\beta,h}]\in\Gamma.
\end{equation}

We first record the $C^0$ estimate for \eqref{eq:penalized}.
\begin{proposition}\label{prop:penalized-c0}
There exists a constant $C_0$ depending on $f$, $\chi$, and $\omega$
such that
\begin{equation*}
\frac{-C_0}{\beta} + \min_M h \le \min_M u_{\beta,h}  \le \max_M u_{\beta,h} \le \max_M h + \frac{C_0}{\beta}.
\end{equation*}
\end{proposition}
\begin{proof}
Let $B=g^{-1}\chi$. Since $\lambda(B(x))$ stays in a compact subset
of $\Gamma$, choose $C_0$ with
$|\log f(\lambda(B(x)))|\leq C_0$ on $M$.
At a maximum point $p_{\max}$ of $u_{\beta,h}$ itself, one has
$\ddc u_{\beta,h}\leq0$. Ellipticity therefore gives
\[
 \beta\bigl(u_{\beta,h}(p_{\max})-h(p_{\max})\bigr)
 \leq\log f(\lambda(B(p_{\max})))\leq C_0,
\]
so $\max_Mu_{\beta,h}\leq\max_Mh+C_0/\beta$.
At a minimum point $p_{\min}$ of $u_{\beta,h}$,
$\ddc u_{\beta,h}\geq0$, and the inequalities reverse:
\[
 \beta\bigl(u_{\beta,h}(p_{\min})-h(p_{\min})\bigr)
 \geq\log f(\lambda(B(p_{\min})))\geq-C_0.
\]
Thus $\min_Mu_{\beta,h}\geq\min_Mh-C_0/\beta$.
This proves the stated estimate with constants independent of the
derivatives of $h$.
\end{proof}

\begin{lemma}\label{lem:penalized-rhs}
There is a constant $C$ depending on $f$, $\omega$, and $\chi$ such that
\begin{equation*}
 e^{\beta(u_{\beta,h}-h)}
 \leq C\bigl(1+\|\partial\bar\partial h\|_{\omega}\bigr).
\end{equation*}
\end{lemma}
\begin{proof}
Let $p_{\max}$ be a maximum point of $u_{\beta,h}-h$.  At this point,
$\ddc u_{\beta,h}\leq\ddc h$.  Hence 
$\lambda[\chi+\ddc h]\in \Gamma$  at $p_{\max}$ and, by ellipticity and
Lemma~\ref{lem:rhs-growth},
\[
 \sup_M e^{\beta(u_{\beta,h}-h)}
 =f(\lambda[\chi+\ddc u_{\beta,h}])(p_{\max})
 \leq f(\lambda[\chi+\ddc h])(p_{\max})
 \leq C\bigl(1+\|\partial\bar\partial h\|_\omega\bigr).
\]
\end{proof}

\begin{proposition}
\label{prop:auxiliary-estimates1}
Fix \(B\geq\osc_M h\) and define
\begin{equation}
 D_h=
 1+\|\nabla h\|_\infty^2+\|\nabla^2h\|_\infty.
 \label{eq:Dh}
\end{equation}
Uniformly for all sufficiently large \(\beta\),
\begin{equation}
 \|\partial\bar\partial u_{\beta,h}\|_\infty
 \leq
 C_B\bigl(1+\|\nabla u_{\beta,h}\|_\infty^2+D_h\bigr).
 \label{eq:aux-estimates}
\end{equation}
\end{proposition}

\begin{proof}
We follow the proof in \cite{S}. In that paper, the right-hand side of the
Hessian equation does not depend on the solution, whereas the right-hand side
of \eqref{eq:penalized} depends on the solution in our setting. Moreover, we do
not assume the existence of the $\mathcal{C}$-subsolution used in \cite{S}.
Fortunately, the sign $\beta>0$ is favorable for our estimate. We must still
ensure that the estimate is uniform in $\beta$. For simplicity, write
$u=u_{\beta,h}$. Let $\lambda=(\lambda_1,\ldots,\lambda_n)$ be the eigenvalues
of $\chi_u$ with respect to $\omega$, ordered so that
$\lambda_1\geq\lambda_2\geq\cdots\geq\lambda_n$. We apply the maximum
principle to the test function
\[
 Q=\log \lambda_1+\phi(|\nabla u|^2)+\psi(u).
\]
Here $\phi$ and $\psi$ are functions to be determined. Let $p$ be a maximum
point of $Q$, and choose normal holomorphic coordinates $z$ centered at $p$.
Let $E$ be a diagonal matrix with $E^{11}=0$ and small entries satisfying
$0<E^{22}<\cdots<E^{nn}<2E^{22}$. To distinguish the deformed tensor from
the background metric, write
\[
 \begin{aligned}
 \widetilde g_{i\bar j}&:=(\chi_u)_{i\bar j}
 =\chi_{i\bar j}+u_{i\bar j},\\
 F^{i\bar j}&:=\frac{\partial F}{\partial \widetilde g_{i\bar j}},
 &
 F^{i\bar j,k\bar\ell}
 &:=\frac{\partial^2F}
 {\partial\widetilde g_{i\bar j}\,\partial\widetilde g_{k\bar\ell}}.
 \end{aligned}
\]
At the chosen diagonal point, we abbreviate $F^{k\bar k}$ by $F^{kk}$.
Define $\widetilde{A}=(\widetilde g_{i\bar j})-E$. At the origin,
$\widetilde{A}$ has eigenvalues
\begin{equation*}
\widetilde{\lambda}_i = \lambda_i - E^{ii}.
\end{equation*}
Thus the eigenvalues of $\widetilde{A}$ are distinct. Differentiating
$F(A_u)=e^{\beta(u-h)}$, we obtain
\begin{equation}\label{eq:first-derivative}
e^{\beta(u-h)}\beta(u-h)_p
=F^{i\bar j}\widetilde g_{i\bar jp}
=F^{kk}\widetilde g_{k\bar kp},
\end{equation}
\begin{equation}
e^{\beta(u-h)}\beta (u-h)_{1\bar{1}} + e^{\beta(u-h)} \beta^2 (u-h)_1 (u-h)_{\bar{1}}= F^{p\bar q,r\bar s} \widetilde g_{p \bar{q}1}\widetilde g_{r\bar{s}\bar{1}} + F^{kk}\widetilde g_{k \bar{k} 1 \bar{1}}.
\end{equation}
Define the linearized operator by $Lw=F^{i\bar j}w_{i\bar j}$.
Equation~\cite[(85)]{S} gives
\begin{equation}
\begin{split}
L(\log \widetilde{\lambda}_1)  \ge & \frac{-F^{p\bar q,r\bar s}\widetilde g_{p \bar{q}1} \widetilde g_{r \bar{s} \bar{1}}  + e^{\beta(u-h)}\beta (u-h)_{1\bar{1}} + e^{\beta(u-h)} \beta^2 (u-h)_1 (u-h)_{\bar{1}}}{\lambda_1} \\
& -\frac{F^{kk}|\widetilde g_{k\bar{1}1}|^2 }{\lambda_1^2} - C_0(\lintr +\lambda_1^{-2} |F^{kk}\widetilde g_{1\bar{1}k }|).
\end{split}
\end{equation}
Here $\lintr=\sum_k F^{kk}$. Unlike in \cite{S}, we cannot
absorb the terms involving $h$ into $-C_0\lintr$. Consequently, we do not
need a positive lower bound for $\lintr$.

Set $K=\|\nabla u\|_{\infty}^2+D_h$. Thus
$Q=\log\widetilde\lambda_1+\phi(|\nabla u|^2)+\psi(u)$.
Here $\phi \triangleq -\frac{1}{2} \log (1-\frac{t}{3K})$ is similar to the function used in \cite{HMW}. However, we change the coefficient of the $t$ term in order to control some terms.
$\phi$ satisfies:
\begin{equation}\label{eq:phi-derivatives}
\frac{1}{6K} < \phi' < \frac{1}{4K} ,\,\,\, \phi''=2\phi'^2>0.
\end{equation}
By Proposition~\ref{prop:penalized-c0}, there exists a constant $C_1$,
depending on $\chi$, $\omega$, $\Gamma$, and $\|h\|_{\infty}$, such that
\begin{equation}
\|u\|_{\infty}\leq C_1.
\end{equation}
We define $\psi\colon[-C_1,C_1]\to\R$ by
\begin{equation*}
\psi (t) = -2A t +\frac{A\tau}{2}t^2.
\end{equation*}
Choose $\tau>0$, depending on $C_1$, sufficiently small that
\begin{equation}\label{eq:psi-estimates}
A \le -\psi' \le 3A, \psi''= A\tau,
\end{equation}
where $A$ is a large constant to be determined. Using
\eqref{eq:first-derivative} and the calculation leading to
\cite[Equation~(95)]{S}, we obtain
\begin{equation*}
\begin{split}
F^{kk} u_{pk \bar{k}} u_{\bar p} &= e^{\beta(u-h)}\beta (u-h)_p u_{\bar p} - F^{kk} \chi_{k \bar k p}u_{\bar p} -T^{k}_{kp}F^{kk}\lambda_k u_{\bar p} + T^q_{kp} F^{kk}\chi_{q \bar k}u_{\bar p}+ F^{kk} R_{k \bar k p}^q u_q u_{\bar p} \\
& \ge - e^{\beta (u-h)} h_p u_{\bar p} - F^{kk} \chi_{k \bar k p}u_{\bar p} -T^{k}_{kp}F^{kk}\lambda_k u_{\bar p} + T^q_{kp} F^{kk}\chi_{q \bar k}u_{\bar p}+ F^{kk} R_{k \bar k p}^q u_q u_{\bar p} \\
& \ge -e^{\beta(u-h)} \beta K - \epsilon_1 F^{kk}\lambda_k^2 - C_{\epsilon_1}\lintr K.
\end{split}
\end{equation*}
Here $\epsilon_1$ is a small constant to be determined, and $T$ and $R$
denote the Chern torsion and curvature of $\omega$. Choosing $\epsilon_1$
sufficiently small and applying the calculation in
\cite[Equations~(96)--(104)]{S} at the maximum point of $Q$, we obtain
\begin{equation*}
\begin{split}
0 \ge LQ \ge & \frac{-F^{p\bar q,r\bar s}\widetilde g_{p \bar{q}1} \widetilde g_{r \bar{s} \bar{1}}  + e^{\beta(u-h)}\beta (u-h)_{1\bar{1}} + e^{\beta(u-h)} \beta^2 (u-h)_1 (u-h)_{\bar{1}}}{\lambda_1} \\
& -\frac{F^{kk}|\widetilde g_{k \bar 1 1}|^2}{\lambda_1^2} + F^{kk}\phi'' |u_{pk}u_{\bar p}+ u_p u_{\bar p k}|^2 + \frac{1}{20K} F^{kk} (\lambda_k^2 + \sum_p |u_{pk}|^2) \\
& + \psi'' F^{kk}u_k u_{\bar k} + \psi' F^{kk}u_{k \bar k} -C_0(\lintr + A \lambda_1^{-1} F^{kk}|u_k|) -2\beta K \phi' e^{\beta(u-h)}.
\end{split}
\end{equation*}
Compared with \cite[Equation~(104)]{S}, the additional terms are those
involving $\beta$:
\begin{equation*}
E_{\beta} \triangleq \frac{ e^{\beta(u-h)}\beta (u-h)_{1\bar{1}} + e^{\beta(u-h)} \beta^2 (u-h)_1 (u-h)_{\bar{1}}}{\lambda_1}-2\beta K \phi' e^{\beta(u-h)}.
\end{equation*}
By \eqref{eq:phi-derivatives}, $|\phi'K|\leq\frac14$. Hence
\begin{equation*}
\begin{split}
E_{\beta} &\ge \frac{e^{\beta(u-h)} \beta (u_{1 \bar 1}+ \chi_{1\bar 1}- \chi_{1\bar 1}-h_{1\bar 1})}{\lambda_1} -2\beta K \phi' e^{\beta (u-h)}= \frac{e^{\beta(u-h)} \beta \Big(\lambda_1(1-2K \phi')-\chi_{1\bar{1}}-h_{1\bar 1} \Big)}{\lambda_1}\\
& \ge  \frac{e^{\beta(u-h)} \beta (\frac{\lambda_1}{2}- K )}{\lambda_1} 
\end{split}
\end{equation*}
If $K\geq\frac{\lambda_1}{4}$, the desired estimate already follows.
Hence we may assume that $K<\frac{\lambda_1}{4}$, which gives
\begin{equation}\label{eq:beta-good-term}
E_{\beta} \ge \frac{\beta}{4} e^{\beta(u-h)} \ge 0.
\end{equation}

Following Hou--Ma--Wu \cite{HMW}, we consider separately the two cases
$-\lambda_n>\delta\lambda_1$ and $-\lambda_n\leq\delta\lambda_1$, where
$\delta>0$ is small.

\textbf{Case 1:} $-\lambda_n>\delta\lambda_1$. By
\eqref{eq:beta-good-term}, the additional term $E_{\beta}$ is nonnegative.
The argument in \cite{S} therefore shows that $K^{-1}\lambda_1$ is bounded.

\textbf{Case 2:} $-\lambda_n\leq\delta\lambda_1$. We cannot directly use
the proof in \cite{S}, because its treatment of this case uses a
$\mathcal{C}$-subsolution and an upper bound for the right-hand side of
\eqref{eq:penalized}. In our setting, we cannot bound that right-hand side by a
constant independent of $h$; the best available estimate is
Lemma~\ref{lem:penalized-rhs}. We therefore use only calculations
(111)--(126) from Case~2 of \cite{S}, which do not use a
$\mathcal{C}$-subsolution. They give
\begin{equation}\label{eq:second-order-case2}
\begin{aligned}
0 \geq {}&E_{\beta}+
\left(
 \frac{1}{2}\psi''
-\delta^2 A^2
-A^2\lambda_1^{-1}
\right)
F^{kk}|u_k|^2
+\frac{1}{32K}F^{kk}\lambda_k^2
+\psi'F^{kk}u_{k\bar{k}} \\
&\quad
-C_0\lintr
-C A^2\delta^{-1}F^{11}K.
\end{aligned}
\end{equation}
By degree-one homogeneity,
\begin{equation*}
F^{kk}u_{k \bar k}=F^{kk} (\chi_{k \bar k} + u_{k \bar k}) - F^{kk}\chi_{k \bar k}= f-F^{kk} \chi_{k\bar k}= e^{\beta(u-h)} - F^{kk}\chi_{k \bar k}. 
\end{equation*}
indeed,
\begin{equation*}
F^{kk} (\chi_{k \bar k} + u_{k \bar k})=f.
\end{equation*}
From \eqref{eq:psi-estimates}, we obtain
\begin{equation}\label{eq:fixed-subsolution}
\psi' F^{kk}u_{k \bar k} \ge -3A e^{\beta(u-h)} -A(-F^{kk}\chi_{k \bar k})= -3A e^{\beta(u-h)} + A F^{kk}(\chi_{k \bar k}- \epsilon_2)+ \epsilon_2 A \lintr.
\end{equation}
Here $\epsilon_2>0$ is chosen sufficiently small that
$\chi-\epsilon_2\omega\in\Gamma$. By
Lemma~\ref{lem:chengxu-pairing},
$F^{kk}(\chi_{k\bar k}-\epsilon_2)\geq0$. Thus
\eqref{eq:fixed-subsolution} becomes
\begin{equation}\label{eq:fixed-subsolution-lower}
\psi' F^{kk}u_{k \bar k}  \ge -3A e^{\beta(u-h)}+ \epsilon_2 A \lintr.
\end{equation}
Substituting \eqref{eq:beta-good-term} and
\eqref{eq:fixed-subsolution} into \eqref{eq:second-order-case2} gives
\begin{equation}\label{eq:second-order-final}
0 \ge (\frac{\beta}{4}-3A) e^{\beta (u-h)} + (\epsilon_2 A - C_0)\lintr + \frac{1}{32K} F^{11}\lambda_1^2 -13A^2 \delta^{-1}F^{11}K + (\frac{1}{2}\psi''- \delta^2 A^2 -A^2 \lambda_1^{-1}) F^{kk}|u_k|^2.
\end{equation}
First choose $A$ sufficiently large that $\epsilon_2A-C_0>0$, and then
choose $\beta$ sufficiently large that $\frac{\beta}{4}-3A>0$. Finally,
choose $\delta$ sufficiently small. If $\lambda_1$ is sufficiently large,
then
\begin{equation*}
 \frac{1}{2}\psi''- \delta^2 A^2 -A^2 \lambda_1^{-1}= \frac{1}{2}\tau A- \delta^2 A^2 -A^2 \lambda_1^{-1} >0.
\end{equation*}
Equation~\eqref{eq:second-order-final} then implies
 \begin{equation*}
 0 \ge  \frac{1}{32K} F^{11}\lambda_1^2 -13A^2 \delta^{-1}F^{11}K.
 \end{equation*}
Thus $\lambda_1\leq CK$ for some constant $C$.
\end{proof}

\begin{proposition}
\label{prop:auxiliary-estimates}
For every $B\geq\osc_M h$ there is a constant $C_B$, depending only on
$B$, $f$, $\omega$, and $\chi$, such that
\begin{equation}
 \|\nabla u_{\beta,h}\|_\infty^2\leq C_B D_h.
\end{equation}
\end{proposition}

\begin{proof}
Suppose that the asserted gradient estimate fails. Then there exists a
sequence of data and solutions satisfying
\begin{equation}\label{eq:blowup-equation}
\log f (\chi+ dd^c u_{\beta_j,h_j})= \beta_j (u_{\beta_j,h_j}-h_j),
\end{equation}
satisfying
\begin{equation}\label{eq:blowup-ratio}
 \frac{D_{h_j}}{N_j}\longrightarrow0,
\end{equation}
where $N_j=\|\nabla u_{\beta_j,h_j}\|_\infty^2$. After discarding finitely
many terms, we may assume that $D_{h_j}/N_j\leq1$. Hence
Proposition~\ref{prop:auxiliary-estimates1} yields
\begin{equation}\label{eq:blowup-hessian}
|\partial \bar \partial u_{\beta_j,h_j}|_{\omega} \le CN_j.
\end{equation}
Define $\widetilde{\omega}=N_j\omega$, and let $p_j$ be a maximum point of
$|\nabla u_{\beta_j,h_j}|$. Choose holomorphic coordinates $z$ centered at
$p_j$ such that
\begin{equation}\label{eq:blowup-coordinates}
\begin{split}
\widetilde{\omega}_{i \bar j} &= \delta_{ij} + O(N_j^{-1}|z|) \\
\chi_{i \bar j} &= O(N_j^{-1}),
\end{split}
\end{equation}
and the coordinates $z$ are defined for $|z|\leq O(N_j^{1/2})$.
Equation~\eqref{eq:blowup-hessian} implies that
\begin{equation}\label{eq:blowup-scaled-hessian}
|\partial \bar \partial u_{\beta_j,h_j}|_{\widetilde{\omega}}\le C.
\end{equation}
Proposition~\ref{prop:penalized-c0} gives a uniform $L^{\infty}$ bound for
$u_{\beta_j,h_j}$. Combining this with
\eqref{eq:blowup-scaled-hessian} and \eqref{eq:blowup-coordinates}, we obtain
the uniform bound
 \begin{equation}\label{eq:blowup-c1alpha}
 \|u_{\beta_j,h_j}\|_{C^{1,\alpha}}<C'
 \end{equation}
 on this ball in these coordinates. We rewrite \eqref{eq:blowup-equation} as
 \begin{equation}\label{eq:blowup-scaled-equation}
N_j f( \lambda[\widetilde{\omega}^{j \bar p}(\chi_{i \bar p}+ (u_{\beta_j,h_j})_{i \bar p} )]) =f(N_j \lambda[\widetilde{\omega}^{j \bar p}(\chi_{i \bar p}+ (u_{\beta_j,h_j})_{i \bar p} )])= e^{\beta_j(u_{\beta_j, h_j}-h_j)}.
 \end{equation}
Here we use the degree-one homogeneity of $f$. Since
$(u_{\beta_j,h_j})_{i\bar p}$ is uniformly bounded, while
$\chi_{i\bar p}$ tends to zero and $\widetilde{\omega}^{j\bar p}$ tends to
the identity matrix, we have
\begin{equation}\label{eq:blowup-eigenvalues}
\lambda[\widetilde{\omega}^{j \bar p} (\chi_{i \bar p}+ (u_{\beta_j,h_j})_{i \bar p})]= \lambda[(u_{\beta_j,h_j})_{i \bar j}]+ O(N_j^{-1}|z|).
\end{equation}
Lemma~\ref{lem:penalized-rhs} gives
\begin{equation}\label{eq:blowup-rhs-upper}
e^{\beta_j(u_{\beta_j,h_j}- h_j)}\le C D_{h_j}.
\end{equation}
Combining \eqref{eq:blowup-scaled-equation},
\eqref{eq:blowup-eigenvalues}, and \eqref{eq:blowup-rhs-upper}, we obtain
\begin{equation*}
f (\lambda[(u_{\beta_j,h_j})_{i \bar j}] +O(N_j^{-1}|z|))\le \frac{C D_{h_j}}{N_j}
\end{equation*}
By \eqref{eq:blowup-c1alpha}, after passing to a subsequence, still denoted by
$(u_{\beta_j,h_j})$, we obtain convergence in $C^{1,\alpha}$ to a function
$v\colon\C^n\to\R$. Equations~\eqref{eq:blowup-ratio} and $D_{h_j}\geq1$
imply that $N_j\to\infty$, so $v$ is defined on all of $\C^n$. By
construction, $|v|$ and $|\nabla v|$ are globally bounded, and
$|\nabla v(0)|=1$.

It suffices to show that $v$ is a $\Gamma$-solution in the sense of
Definition~\ref{def:gamma-solution}, since this contradicts
Lemma~\ref{lem:gamma-liouville}.

First, suppose that a $C^2$ function $\psi$ touches $v$ from above at
$z_0$. We show that
$\lambda(\psi_{i\bar j}(z_0))\in\overline{\Gamma}$. By the construction of
$v$, for every $\epsilon>0$ and all sufficiently large $j$, there exist
$z_1\in B_{\epsilon}(z_0)$ and a constant $a$ such that
\begin{equation*}
\psi+ \epsilon |z-z_0|^2 +a \ge u_{\beta_j,h_j} \text{ on } B_1(z_0), \text{ with equality at }z_1.
\end{equation*}
This implies that $\psi_{i \bar j}(z_1) + \epsilon \delta_{ij} \ge (u_{\beta_j,h_j})_{i\bar j}(z_1)$. Since
\begin{equation*}
\lambda[\widetilde{\omega}^{j \bar p}(\chi_{i \bar p }+ (u_{\beta_j,h_j})_{i \bar p})]=\frac{1}{N_j}\lambda[\omega^{j \bar p}(\chi_{i \bar p }+ (u_{\beta_j,h_j})_{i \bar p})] \in \frac{1}{N_j}\Gamma = \Gamma.
\end{equation*}
Since $\Gamma+\Gamma_n\subset\Gamma$ and
$\psi_{i\bar j}(z_1)+\epsilon\delta_{ij}\geq
(u_{\beta_j,h_j})_{i\bar j}(z_1)$, it follows that
\begin{equation}\label{eq:blowup-upper-admissible}
\lambda[\widetilde{\omega}^{j \bar p}(\chi_{i \bar p }+ \epsilon \delta_{ip} + \psi_{i \bar p})] (z_1) \in \Gamma.
\end{equation}
As in \eqref{eq:blowup-eigenvalues},
\begin{equation}\label{eq:blowup-test-eigenvalues}
\lambda[\widetilde{\omega}^{j \bar p} (\chi_{i \bar p}+ \psi_{i \bar p})]= \lambda[\psi_{i \bar j}]+ O(N_j^{-1}|z|).
\end{equation}
Letting $\epsilon$ tend to zero in \eqref{eq:blowup-upper-admissible}, we obtain
\begin{equation*}
\lambda[\psi_{i \bar p}](z_0) \in \overline{\Gamma}.
\end{equation*}

Suppose now that a $C^2$ function $\psi$ touches $v$ from below at $z_0$.
We show that
$\lambda[\psi_{i\bar j}(z_0)]\in\R^n\setminus\Gamma$. Arguing as above,
for every $\epsilon>0$ and all sufficiently large $j$, there exist
$z_1\in B_{\epsilon}(z_0)$ and a constant $a$ such that
\begin{equation*}
\psi- \epsilon |z-z_0|^2 + a \le u_{\beta_j,h_j} \text{ on } B_1(z_0), \text{ with equality at }z_1.
\end{equation*}
Consequently,
$\psi_{i\bar j}(z_1)-\epsilon\delta_{i\bar j}\leq
(u_{\beta_j,h_j})_{i\bar j}(z_1)$. If
$\lambda[\psi_{i\bar j}(z_1)-3\epsilon\delta_{ij}]\in\Gamma$, then
$\lambda((u_{\beta_j,h_j})_{i\bar j})\in
\Gamma+2\epsilon\mathbf{1}$. Set
\[
 \Lambda_j:=
 \lambda\!\left[
 \widetilde{\omega}^{j\bar p}
 \bigl(\chi_{i\bar p}+(u_{\beta_j,h_j})_{i\bar p}\bigr)
 \right](z_1).
\]
By \eqref{eq:blowup-eigenvalues}, after increasing $j$ if necessary,
\begin{equation}\label{eq:blowup-cone-shift}
 \Lambda_j\in\Gamma+\epsilon\mathbf{1}.
\end{equation}
Choose a sequence $\{b_j\}$ satisfying
\begin{equation*}
\lim_{j \rightarrow \infty} b_j =\infty,
\qquad
\lim_{j \rightarrow \infty} \frac{b_j D_{h_j}}{N_j}=0.
\end{equation*}
By Lemma~\ref{lem:cone-translation} and \eqref{eq:blowup-cone-shift}, we have
\begin{equation*}
 \lim_{j\to\infty}f(b_j\Lambda_j)=\infty.
\end{equation*}
On the other hand, \eqref{eq:blowup-scaled-equation}, the homogeneity of
$f$, and \eqref{eq:blowup-rhs-upper} give
\begin{equation*}
 f(b_j\Lambda_j)
 =b_jf(\Lambda_j)
 =\frac{b_j}{N_j}e^{\beta_j(u_{\beta_j,h_j}-h_j)}(z_1)
 \leq\frac{CD_{h_j}b_j}{N_j}\longrightarrow0,
\end{equation*}
a contradiction.
Thus, we cannot have
$\lambda(\psi_{i\bar{j}}(z_1))\in\Gamma+3\epsilon\mathbf{1}$. Letting
$\epsilon\to0$, we have $z_1\to z_0$ and
$\lambda(\psi_{i\bar{j}}(z_0))\in\R^n\setminus\Gamma$. This completes the
proof that $v$ is a $\Gamma$-solution.
\end{proof}

\begin{proposition}\label{prop:uniform-convergence}
For every $h\in C^\infty(M)$,
\begin{equation*}
\lim_{\beta \rightarrow \infty}
\|u_{\beta,h}-P_{\Gamma,\chi}(h)\|_{\infty}=0.
\end{equation*}
\end{proposition}
\begin{proof}
Let
\[
 M_h=C\bigl(1+\|\partial\bar\partial h\|_\infty\bigr),
 \qquad
 \eta_\beta=\frac1\beta\log\max\{1,M_h\}.
\]
The maximum-point argument in Lemma~\ref{lem:penalized-rhs} gives
\(u_{\beta,h}-h\leq\eta_\beta\).  Hence
\(u_{\beta,h}-\eta_\beta\) is an admissible competitor below \(h\), and
therefore
\begin{equation}
 u_{\beta,h}\leq P_{\Gamma,\chi}(h)+\eta_\beta.
 \label{eq:penalized-envelope-upper}
\end{equation}

For the opposite inequality, let \(v\leq h\) be any viscosity-admissible
competitor in the definition of the envelope, set
\(c_h=\inf_M h-1\), and define
\[
 v_\delta=(1-\delta)v+\delta c_h.
\]
At every upper test, the total Hessian of \(v_\delta\) is the convex
combination of the total Hessian of \(v\) and \(g^{-1}\chi\).  Thus
\(v_\delta\) is strictly admissible.  Concavity, homogeneity, and the
continuous extension of \(f\) to \(\overline\Gamma\) give, in the viscosity
sense,
\begin{equation}
 f(\lambda[\chi+\ddc v_\delta])
 \geq\delta c_\chi,
 \qquad
 c_\chi:=\min_Mf(\lambda(g^{-1}\chi))>0.
 \label{eq:strictified-competitor}
\end{equation}
Moreover \(v_\delta-h\leq-\delta\).  Consequently
\[
 f(\lambda[\chi+\ddc v_\delta])
 e^{-\beta(v_\delta-h)}
 \geq\delta c_\chi e^{\beta\delta}.
\]
The penalized solution satisfies the same left-hand side with value one.
If \(\delta c_\chi e^{\beta\delta}>1\), the viscosity comparison
principle gives \(v_\delta\leq u_{\beta,h}\).  Taking the supremum over
all competitors and then the upper-semicontinuous regularization yields
\begin{equation}
 (1-\delta)P_{\Gamma,\chi}(h)+\delta c_h
 \leq u_{\beta,h}.
 \label{eq:penalized-envelope-lower}
\end{equation}
For example, \(\delta_\beta=2(\log\beta)/\beta\) satisfies the required
strict inequality for all large \(\beta\).  Since the envelope is bounded
between a fixed constant competitor and \(h\), equations
\eqref{eq:penalized-envelope-upper}--\eqref{eq:penalized-envelope-lower}
give
\begin{equation}
 \|u_{\beta,h}-P_{\Gamma,\chi}(h)\|_\infty
 \leq C_h\frac{\log\beta}{\beta}\longrightarrow0.
 \label{eq:uniform-penalized-convergence}
\end{equation}
\end{proof}

\begin{proof}[Proof of Theorem~\ref{thm:penalized-c2}]
Combine Propositions~\ref{prop:penalized-c0},
\ref{prop:auxiliary-estimates1}, and~\ref{prop:auxiliary-estimates}.
\end{proof}

\begin{proof}[Proof of Corollary~\ref{cor:penalized-c1alpha}]
Lemma~\ref{lem:determinant-domination} verifies the determinant-domination
hypothesis in \cite[Theorem~3.1(1)]{ChengXu2025}, so the smooth penalized
solutions exist.  Proposition~\ref{prop:uniform-convergence} gives uniform convergence
to the envelope, and Theorem~\ref{thm:penalized-c2} gives uniform bounds for
$\nabla u_{\beta,h}$ and $\partial\bar\partial u_{\beta,h}$.
In particular,
\[
 \Delta_\omega^{\mathrm{Ch}}u_{\beta,h}
 =\operatorname{tr}_\omega(\ddc u_{\beta,h})
\]
is uniformly bounded.  The global $W^{2,p}$ estimate for the fixed
uniformly elliptic operator $\Delta_\omega^{\mathrm{Ch}}$, followed by
Sobolev embedding, gives precompactness in $C^{1,\alpha}$ for every
$\alpha<1$.  Every subsequential limit is the same uniform limit
$P_{\Gamma,\chi}(h)$, so the full family converges in $C^{1,\alpha}$.
\end{proof}

We now prove Corollary~\ref{cor:envelope-regularity}, assuming
Theorem~\ref{thm:canonical-operator}.
\begin{proof}[Proof of Corollary~\ref{cor:envelope-regularity}]
Let $f_{\Gamma}$ be the function given by
Theorem~\ref{thm:canonical-operator}. For each $\beta$, let $u_{\beta,h}$
solve
\begin{equation}
 f_{\Gamma}(\lambda[\chi+\ddc u_{\beta,h}])
 =e^{\beta(u_{\beta,h}-h)},
 \qquad \lambda[\chi+\ddc u_{\beta,h}]\in\Gamma,
\end{equation}
Combining Theorem~\ref{thm:penalized-c2} with
Corollary~\ref{cor:penalized-c1alpha} proves the result.
\end{proof}

\section{The canonical cone operator}

In this section we prove Theorem~\ref{thm:canonical-operator}.  Let $\Gamma$ be a
cone satisfying Assumption~\ref{ass1}(1).

\subsection{The trace-cone case}

When
\[
\Gamma= \Gamma_1=\left\{\lambda:\sum_i\lambda_i>0\right\},
\]
we define
\begin{equation}
 f_{\Gamma_1}(\lambda)=\frac{1}{n}\sum_i\lambda_i.
 \label{eq:trace-cone-operator}
\end{equation}
\begin{lemma}\label{lem:trace-cone}
$(\Gamma_1,f_{\Gamma_1})$ satisfies Assumptions~\ref{ass1}--\ref{ass3},
with determinant constant $n^{-n}$.
\end{lemma}
\begin{proof}
All assertions are immediate from
$(f_{\Gamma_1})_i=1/n$ and
$f_{\Gamma_1}=0$ on $\partial\Gamma_1$.
\end{proof}

\subsection{The proper-cone case}

The dual cone of $\overline{\Gamma}$ is
\[
\overline{\Gamma}^{\,*}=\{y\in\R^n:\langle x,y\rangle\geq0
      \text{ for every }x\in\overline{\Gamma}\}.
\]
Because $\Gamma_n\subset\Gamma$, one has
\begin{equation}\label{eq:dual-positive-orthant}
 \overline{\Gamma}^{\,*}\subset\overline{\Gamma_n}.
\end{equation}
Since $\Gamma \neq \Gamma_1$ and $\Gamma$ is symmetric, its dual cone has nonempty interior.

The Koszul--Vinberg characteristic function is
\[
\Phi_{\overline{\Gamma}}(\lambda)=\int_{\overline{\Gamma}^{\,*}}e^{-\langle\lambda,y\rangle}\,dy,
\qquad \lambda\in\Gamma.
\]
After fixing $e=(1,\ldots,1)$ as the normalization point, define
\begin{equation}
f_\Gamma(\lambda)
=\left(\frac{\Phi_{\overline{\Gamma}}(e)}{\Phi_{\overline{\Gamma}}(\lambda)}\right)^{1/n}.
\label{eq:characteristic-operator}
\end{equation}
For every compact $L\Subset\Gamma$ there is $a_L>0$ such that
$\langle\lambda,y\rangle\geq a_L|y|$ for
$\lambda\in L$ and $y\in\overline{\Gamma}^{\,*}$.  Hence the integral is
finite and all of its derivatives converge locally uniformly on $\Gamma$.

Set
\[
H=\{p\in\R^n:\langle p,e\rangle=1\},
\qquad
K=\overline{\Gamma}^{\,*}\cap H,
\]
and let $d\sigma$ denote Euclidean $(n-1)$-dimensional Hausdorff measure
on $H$. Since $\overline{\Gamma}^*$ has nonempty interior,  $K$
is a compact convex body of positive $(n-1)$-dimensional measure.

\begin{proposition}
For every $\lambda\in\Gamma$,
\begin{equation}\label{eq:phi-slice}
\Phi_{\overline{\Gamma}}(\lambda)
=\frac{\Gamma(n)}{\sqrt n}
 \int_K\langle p,\lambda\rangle^{-n}\,d\sigma(p).
\end{equation}
Consequently, if
\begin{equation}\label{eq:I-definition}
I(\lambda)=\int_K\langle p,\lambda\rangle^{-n}\,d\sigma(p),
\end{equation}
then
\begin{equation}\label{eq:f-I-representation}
f_\Gamma(\lambda)
=\left(\frac{I(e)}{I(\lambda)}\right)^{1/n}.
\end{equation}
Moreover, $I(e)=\sigma(K)$.
\end{proposition}

\begin{proof}
First observe that
\[
\langle e,y\rangle>0
\qquad\text{for every }y\in \overline{\Gamma}^{\,*}\setminus\{0\}.
\]
Indeed, if $\langle e,y\rangle=0$ for some nonzero $y\in \overline{\Gamma}^{\,*}$, then,
because $e\in\operatorname{Int}\overline{\Gamma}$, one has
$e-ty\in\overline{\Gamma}$ for all sufficiently small $t>0$.  This would give
\[
\langle e-ty,y\rangle=-t|y|^2<0,
\]
contradicting $y\in \overline{\Gamma}^{\,*}$.

For $y\in \overline{\Gamma}^{\,*}\setminus\{0\}$, define
\[
r=\langle e,y\rangle,
\qquad
p=\frac{y}{\langle e,y\rangle}.
\]
Then $r>0$, $p\in K$, and $y=rp$.  Conversely, $rp\in \overline{\Gamma}^{\,*}$ whenever
$r>0$ and $p\in K$.  Thus the map
\[
F:(0,\infty)\times K\longrightarrow \overline{\Gamma}^{\,*}\setminus\{0\},
\qquad F(r,p)=rp,
\]
is one-to-one and onto.

We now compute its Jacobian.  Let
$\tau_1,\ldots,\tau_{n-1}$ be an orthonormal basis of the tangent space
of $H$, which is $e^\perp$.  Then
\[
\partial_rF=p,
\qquad
D_{\tau_j}F=r\tau_j.
\]
The unit normal to $H$ is $N=e/|e|$.  Therefore
\begin{align*}
J_F(r,p)
&=\left|\det[p,r\tau_1,\ldots,r\tau_{n-1}]\right| \\
&=r^{n-1}|\langle p,N\rangle| \\
&=\frac{r^{n-1}}{|e|}
=\frac{r^{n-1}}{\sqrt n},
\end{align*}
where we used $\langle p,e\rangle=1$ and $|e|=\sqrt n$.
Consequently,
\begin{equation}\label{eq:cone-volume-element}
dy=\frac{r^{n-1}}{\sqrt n}\,dr\,d\sigma(p).
\end{equation}

Substituting $y=rp$ into the definition of $\Phi_{\overline{\Gamma}}$, we obtain
\begin{align*}
\Phi_{\overline{\Gamma}}(\lambda)
&=\frac1{\sqrt n}
  \int_K\int_0^\infty
  e^{-r\langle p,\lambda\rangle}r^{n-1}\,dr\,d\sigma(p).
\end{align*}
Since $\lambda\in\operatorname{Int}\overline{\Gamma}$ and
$p\in \overline{\Gamma}^{\,*}\setminus\{0\}$, one has
$\langle p,\lambda\rangle>0$.  The Gamma-integral identity
\[
\int_0^\infty e^{-ar}r^{n-1}\,dr
=a^{-n}\Gamma(n),
\qquad a>0,
\]
therefore gives
\[
\Phi_{\overline{\Gamma}}(\lambda)
=\frac{\Gamma(n)}{\sqrt n}
 \int_K\langle p,\lambda\rangle^{-n}\,d\sigma(p),
\]
which proves \eqref{eq:phi-slice}.

The same dimensional constant occurs at $\lambda=e$, so it cancels in the
quotient.  More explicitly,
\[
\frac{\Phi_{\overline{\Gamma}}(e)}{\Phi_{\overline{\Gamma}}(\lambda)}
=\frac{I(e)}{I(\lambda)}.
\]
Hence \eqref{eq:f-I-representation} follows.  Finally,
$\langle p,e\rangle=1$ on $K$, and therefore
\[
I(e)=\int_K1\,d\sigma=\sigma(K).
\]
\end{proof}

\begin{remark}
If
\[
d\mu=\frac{d\sigma}{\sigma(K)},
\]
then \eqref{eq:f-I-representation} becomes
\[
f_\Gamma(\lambda)
=\left(
  \int_K\langle p,\lambda\rangle^{-n}\,d\mu(p)
 \right)^{-1/n}.
\]
The exponent $n$ comes from the factor $r^{n-1}$ in the
$n$-dimensional radial Jacobian.
\end{remark}

\begin{lemma}\label{lem:canonical-concavity}
$f_{\Gamma}$ is concave and homogeneous of degree one.
\end{lemma}
\begin{proof}
The change of variables $y'=ty$ gives
$\Phi_{\overline\Gamma}(t\lambda)=t^{-n}
\Phi_{\overline\Gamma}(\lambda)$, so $f_\Gamma$ is homogeneous of
degree one.  A positive, degree-one homogeneous, log-concave function is
concave.
Indeed, write
\[
a=f_\Gamma(\lambda),
\qquad
b=f_\Gamma(\mu),
\qquad
L=(1-\theta)a+\theta b,
\]
and set
\[
\alpha=\frac{(1-\theta)a}{L},
\qquad
\beta=\frac{\theta b}{L}.
\]
Then $\alpha+\beta=1$ and
\[
\frac{(1-\theta)\lambda+\theta\mu}{L}
=\alpha\frac{\lambda}{a}+\beta\frac{\mu}{b}.
\]
Homogeneity gives
\[
f_\Gamma\left(\frac{\lambda}{a}\right)
=f_\Gamma\left(\frac{\mu}{b}\right)=1.
\]
Log-concavity therefore implies
\[
f_\Gamma\left(
 \frac{(1-\theta)\lambda+\theta\mu}{L}
\right)\geq1.
\]
Using degree-one homogeneity once more, we conclude that
\[
f_\Gamma((1-\theta)\lambda+\theta\mu)
\geq(1-\theta)f_\Gamma(\lambda)+\theta f_\Gamma(\mu).
\]

It remains to prove that $f_{\Gamma}$ is log-concave.
For $0<\theta<1$, H\"older's inequality gives
\begin{align*}
\Phi_{\overline{\Gamma}}((1-\theta)\lambda+\theta\mu)
&=\int_{\overline{\Gamma}^{\,*}}
  \left(e^{-\langle\lambda,y\rangle}\right)^{1-\theta}
  \left(e^{-\langle\mu,y\rangle}\right)^\theta\,dy \\
&\leq
  \Phi_{\overline{\Gamma}}(\lambda)^{1-\theta}\Phi_{\overline{\Gamma}}(\mu)^\theta.
\end{align*}
Hence $\log\Phi_{\overline{\Gamma}}$ is convex, and therefore
\[
\log f_\Gamma
=\frac1n\log\Phi_{\overline{\Gamma}}(e)-\frac1n\log\Phi_{\overline{\Gamma}}
\]
is concave.  Thus $f_\Gamma$ is log-concave, completing the proof.
\end{proof}

\begin{lemma}\label{lem:canonical-ellipticity}
$f_{\Gamma}$ satisfies:
\begin{equation*}
(f_{\Gamma})_i(\lambda)>0 \text{ on } \Gamma,\,\,\,f_{\Gamma}|_{\partial \Gamma}=0.
\end{equation*}
\end{lemma}
\begin{proof}
The local convergence noted after \eqref{eq:characteristic-operator}
shows that $\Phi_{\overline\Gamma}$ and $f_\Gamma$ are smooth and
positive in $\Gamma$.  Differentiating
\eqref{eq:characteristic-operator}, we obtain
\begin{equation}
 (f_{\Gamma})_i(\lambda)
 =\frac{f_{\Gamma}(\lambda)}{n\Phi_{\overline{\Gamma}}(\lambda)}
   \int_{\overline{\Gamma}^{\,*}}y_i e^{-\lambda\cdot y}\,dy.
 \label{eq:characteristic-first-derivative}
\end{equation}
By \eqref{eq:dual-positive-orthant}, $y_i\geq0$ on
$\overline{\Gamma}^{\,*}$, and therefore
\[
 (f_{\Gamma})_i\geq0.
\]
In fact, since \(\overline{\Gamma}^{\,*}\) is full-dimensional, \(y_i>0\) on a set of positive
measure, and consequently \((f_{\Gamma})_i>0\).

It remains to check the boundary value.  Suppose
$\lambda_j\to\lambda_0\in\partial\Gamma$.  By separation, there is
$0\neq y_0\in\overline{\Gamma}^{\,*}$ such that
\[
 \lambda_0\cdot y_0=0.
\]
Choose $q\in\operatorname{Int}\overline{\Gamma}^{\,*}$ and a small
$(n-1)$-dimensional disk $D\subset q+y_0^\perp$ such that
$D\subset\operatorname{Int}\overline{\Gamma}^{\,*}$.  Convexity of the
cone implies $sy_0+D\subset\overline{\Gamma}^{\,*}$ for $s\geq0$.
The map $(s,z)\mapsto sy_0+z$ has a fixed positive Jacobian, and hence
\[
 \Phi_{\overline{\Gamma}}(\lambda_j)
 \geq c\int_D\int_0^\infty
 e^{-\lambda_j\cdot(sy_0+z)}\,ds\,d\sigma_D(z)
 \geq\frac{c'}{\lambda_j\cdot y_0}
 \longrightarrow+\infty.
\]
This is the boundary blow-up of the classical characteristic barrier; see
also \cite{Guler1996}.
Consequently,
\begin{equation}\label{eq:characteristic-boundary-zero}
 f_{\Gamma}(\lambda_j)\longrightarrow0.
\end{equation}
Symmetry follows from the permutation
invariance of \(\overline{\Gamma}^{\,*}\) and Lebesgue measure.
\end{proof}

\begin{lemma}\label{lem:canonical-structure}
Let $F_\Gamma(A):=f_\Gamma(\lambda(A))$. There is $\gamma_\Gamma>0$
such that its spectral linearization satisfies
\begin{equation}\label{eq:canonical-det-linearization}
 \det DF_\Gamma(A)\geq\gamma_\Gamma
 \qquad\text{for all }\lambda(A)\in\Gamma.
\end{equation}
\end{lemma}
\begin{proof}
For $\lambda\in\Gamma_n$, equation
\eqref{eq:dual-positive-orthant} gives
\[
 \Phi_{\overline{\Gamma}}(\lambda)
 \leq\int_{\R_+^n}e^{-\lambda\cdot y}\,dy
 =\frac{1}{\lambda_1\cdots\lambda_n}.
\]
It follows that $f_{\Gamma}$ satisfies the determinant majorization condition:
\begin{equation}
 f_{\Gamma}(\lambda)
 \geq\Phi_{\overline{\Gamma}}(e)^{1/n}
       (\lambda_1\cdots\lambda_n)^{1/n},
 \qquad \lambda\in\Gamma_n.
 \label{eq:canonical-determinant-majorization}
\end{equation}
Equation~\eqref{eq:canonical-det-linearization} now follows from
\cite[Lemma~4]{GuoPhongTong2023}.
\end{proof}

In conclusion, define
\begin{equation}\label{eq:canonical-operator}
 f_\Gamma(\lambda)=
 \begin{cases}
 \displaystyle
 \left(\frac{\Phi_{\overline{\Gamma}}(e)}{\Phi_{\overline{\Gamma}}(\lambda)}\right)^{1/n},
 & \Gamma\neq\Gamma_1,\\[1.4ex]
 \displaystyle \frac1n\sum_i\lambda_i,
 & \Gamma=\Gamma_1,
 \end{cases}
\end{equation}

\begin{proof}[Proof of Theorem~\ref{thm:canonical-operator}]
Let $f_\Gamma$ be defined by \eqref{eq:canonical-operator}. In the
$\Gamma_1$ case, the theorem follows from Lemma~\ref{lem:trace-cone}. In the
proper-cone case,
Lemmas~\ref{lem:canonical-concavity}, \ref{lem:canonical-ellipticity}, and
\ref{lem:canonical-structure} verify Assumptions~\ref{ass1}, \ref{ass2}, and
\ref{ass3}, respectively.
\end{proof}

\section{Proof of Theorem~\ref{thm:regularization}}

\begin{proof}[Proof of Theorem~\ref{thm:regularization}]
By the definition of $u_t$, we have that $u_t\in C^{\infty}(M)\cap SH_{\chi,\Gamma}(M,\omega)$.

Proposition~\ref{prop:heat-envelope} gives
\begin{equation}\label{eq:heat-upper}
 u \leq h_t.
\end{equation}
Since $u\in USC(X)$, $u\leq h_t$, and
$\lambda[\chi+dd^c u]\in\overline{\Gamma}$ in the viscosity sense, the
definition of the envelope gives
\begin{equation}\label{eq:sandwich1}
 u\leq P_{\Gamma,\chi}(h_t)\leq h_t.
\end{equation}
Combining \eqref{eq:sandwich1} and \eqref{eq:L1-rate}, we obtain
\begin{equation}\label{eq:L1-rate1}
 \|P_{\Gamma,\chi}(h_t)-u\|_{L^1(dV_\omega)}\leq\|h_t-u\|_{L^1(dV_\omega)}
 \leq C_{\omega,\chi}t.
\end{equation}

Combining this with
$\|u_t-P_{\Gamma,\chi}(h_t)\|_{\infty}\leq t$, we obtain
\begin{equation}\label{eq:main-L1-comparator}
 u_t-u\geq-t,
 \qquad
 \|(u_t-u)^+\|_{L^1(dV_\omega)}\leq Ct.
\end{equation}
Proposition~\ref{prop:heat-envelope} gives
\[
 \|\nabla h_t\|_\infty^2+\|\nabla^2h_t\|_\infty
 \leq Ct^{-1}.
\]
Moreover $\osc_M h_t\leq\osc_M u+Ct$ is uniformly bounded. Applying
Theorem~\ref{thm:penalized-c2} to $u_t$ therefore gives
\begin{equation}\label{eq:regularization-gradient}
 \|\nabla u_t\|_\infty\leq Ct^{-1/2},\,\,\,\, \| dd^c u_t \|_{\infty}\le C t^{-1}.
\end{equation}

Equation~\eqref{eq:sandwich} also implies
\[
 \|u_t\|_{\infty}\leq C,
\]
where $C$ depends on $\|u\|_{\infty}$.
\end{proof}

\section{Stability}
Set
\[
 dV:=\frac{\omega^n}{\int_M\omega^n},
 \qquad
 \Vol(E):=\int_E dV
\]
for every Borel set $E\subset M$. Thus $dV(M)=1$.
Since \(M\) is compact and \(\Gamma\) is open, there is a fixed number
\(0<\vartheta_\chi\leq1\) such that
\begin{equation} \label{eq:detailed-chi-margin}
 \lambda [\chi-\tau \omega](x)\in\Gamma
 \quad\text{for every }x\in M\text{ and }0\leq\tau\leq\vartheta_\chi.
\end{equation}
For \(0<\delta<1\), \(s\geq0\), and an exactly
\(\Gamma\)-admissible \(v\in C^2(M)\), meaning precisely
\begin{equation}
 \lambda(A_v)
 =\lambda[\chi+\ddc v]\in\Gamma
 \quad\text{on }M,
 \label{eq:detailed-v-exact-admissibility}
\end{equation}
put
\[
 q_\delta=(1-\delta)v-u,
 \qquad
 \Omega_{\delta,s}=\{q_\delta>s\},
 \qquad
 w_{\delta,s}=(q_\delta-s)^+,
\]
and
\begin{equation}
 \mathcal A_{\delta,s,\kappa}
 =\left(\int_M
       \bigl[w_{\delta,s}e^{nG}\bigr]^\kappa\,dV
   \right)^{1/\kappa}
 =\|w_{\delta,s}e^{nG}\|_{L^\kappa(dV)},
 \label{eq:detailed-stability-A}
\end{equation}
where \(1<\kappa<p_0\) will be chosen sufficiently close to one. Here
\(e^{nG}=(e^{G(x)})^n\) is an ordinary scalar power. We prove the following
estimate of Guo--Phong--Tong type \cite{GuoPhongTong2023}.

\begin{proposition}
\label{prop:detailed-viscosity-trudinger}
There are constants \(\beta_0,C>0\), depending only on the fixed
Hermitian metric, the fixed form \(\chi\) (in particular its uniform
margin \(\vartheta_\chi\)), \(n\), \(\kappa\), and the structural
constants, such that, whenever
\(\mathcal A_{\delta,s,\kappa}>0\),
\begin{equation}
 \int_M
 \exp\!\left(
 \beta_0
 \frac{w_{\delta,s}^{(n+1)/n}}
      {\mathcal A_{\delta,s,\kappa}^{1/n}}
 \right)dV
 \leq
 C\exp\!\left(
 C\delta^{-(n+1)}\mathcal A_{\delta,s,\kappa}
 \right).
 \label{eq:detailed-trudinger}
\end{equation}
If \(\mathcal A_{\delta,s,\kappa}=0\), then
\(w_{\delta,s}\equiv0\).
\end{proposition}

\begin{proof}
Write \(w=w_{\delta,s}\) and
\(\mathcal A=\mathcal A_{\delta,s,\kappa}>0\).
Since \(we^{nG}\) is continuous and nonnegative, choose smooth strictly
positive functions \(\widehat\rho_j\) such that
\[
 \widehat\rho_j\geq we^{nG},
 \qquad
 \widehat\rho_j\longrightarrow we^{nG}
 \quad\hbox{uniformly on }M.
\]
Set
\begin{equation}
 \mathcal A_j=\|\widehat\rho_j\|_{L^\kappa(dV)},
 \qquad
 \rho_j=\frac{\widehat\rho_j}{\mathcal A_j}.
 \label{eq:detailed-density}
\end{equation}
Then \(\mathcal A_j\to\mathcal A\) and
\(\|\rho_j\|_{L^\kappa}=1\).  The Hermitian Monge--Amp\`ere theorem
\cite{TW} gives a smooth
function \(\psi_j\) and a constant \(b_j\in\mathbb R\) satisfying
\begin{equation}
 (\omega+\ddc\psi_j)^n=e^{b_j}\rho_j\omega^n,
 \qquad
 \sup_M\psi_j=0.
 \label{eq:detailed-MA}
\end{equation}
The scalar \(b_j\) cannot in general be set equal to zero: on a Hermitian
manifold the integral of \((\omega+\ddc\psi)^n\) is not fixed by the
cohomology class.  This is the first point at which the K\"ahler proof must
be changed.
The lower-bound part of \cite[Lemma~5.9]{KolodziejNguyen2015} gives
\begin{equation}
 b_j\geq-C_{\mathrm{MA}}(M,\omega,\kappa).
 \label{eq:detailed-b-lower}
\end{equation}
This is exactly why the density was normalized in \(L^\kappa\), rather
than in \(L^1\). This is a trick used in \cite{ChengXu2025}. No upper bound for \(b_j\) will be used.

Define the auxiliary positive-definite endomorphism
\begin{equation}
 T_j:=g^{-1}(\omega+\ddc\psi_j)
     =I+g^{-1}\ddc\psi_j.
 \label{eq:detailed-Tj}
\end{equation}
Then $\det T_j=e^{b_j}\rho_j$.  Notice that $T_j$ is built from
$\omega$; it is not the equation endomorphism
$g^{-1}(\chi+\ddc\psi_j)$.  Let
\[
 \alpha=\frac{n}{n+1},
 \qquad
 K_\kappa=
 \left(\alpha n \gamma_0^{1/n}
 e^{-C_{\mathrm{MA}}/n}\right)^{-1},
\]
\begin{equation}
 \varepsilon_j=(2K_\kappa^n\mathcal A_j)^{-1/(n+1)},
 \qquad
 \Lambda_j=
 \left(
 \frac{10\alpha}
      {\vartheta_\chi\delta\varepsilon_j}
 \right)^{n+1},
 \qquad
 R_j=-\psi_j+\Lambda_j.
 \label{eq:detailed-parameters}
\end{equation}
Because \(\psi_j\leq0\), one has \(R_j\geq\Lambda_j>0\).

We claim that
\begin{equation}
 \Theta_j:=\varepsilon_jw-R_j^\alpha\leq0
 \qquad\hbox{on }M.
 \label{eq:detailed-theta}
\end{equation}
Assume to the contrary that \(\Theta_j\) has a positive maximum at
\(x_j\).  Positivity implies \(w(x_j)>0\).  Since
\(q_\delta-s\) is continuous, it is positive in a neighborhood of
\(x_j\), and there
\[
 w=(1-\delta)v-u-s.
\]
Define the smooth function
\begin{equation}
 P_j=(1-\delta)v-s-\varepsilon_j^{-1}R_j^\alpha.
 \label{eq:detailed-lower-test}
\end{equation}
Near \(x_j\),
\[
 \Theta_j=\varepsilon_j(P_j-u).
\]
If \(m_j=(P_j-u)(x_j)\), maximality gives
\(P_j-m_j\leq u\) near \(x_j\), with equality at \(x_j\).
Thus \(P_j-m_j\) is a genuine smooth lower test for \(u\).  Since a
constant does not change the complex Hessian, we continue to denote the
test by \(P_j\).

Put
\begin{equation}
 c_j=\alpha\varepsilon_j^{-1}R_j^{-1/(n+1)}.
 \label{eq:detailed-cj}
\end{equation}
The definition of \(\Lambda_j\) gives
\begin{equation}
 c_j
 \leq\alpha\varepsilon_j^{-1}
       \Lambda_j^{-1/(n+1)}
 =\frac{\vartheta_\chi\delta}{10}.
 \label{eq:detailed-cj-bound}
\end{equation}
Since \(R_j=-\psi_j+\Lambda_j\), direct differentiation gives
\begin{equation}
 \ddc\!\left(-\varepsilon_j^{-1}R_j^\alpha\right)
 =c_j\ddc\psi_j
 +\frac{\alpha(1-\alpha)}{\varepsilon_j}
  R_j^{\alpha-2}
  \sqrt{-1}\,\partial\psi_j\wedge\bar\partial\psi_j.
 \label{eq:detailed-chain-rule}
\end{equation}
Consequently
\begin{equation}
 A_{P_j}
 =(1-\delta)A_v
  +\delta\left(B-\frac{c_j}{\delta}I\right)
  +c_j T_j+R_j^+,
 \label{eq:detailed-decomposition}
\end{equation}
where
\begin{equation}
 R_j^+
 =\frac{\alpha(1-\alpha)}{\varepsilon_j}
 R_j^{\alpha-2}
 g^{-1}\!\left(
 \sqrt{-1}\,\partial\psi_j\wedge\bar\partial\psi_j
 \right)\geq0.
 \label{eq:detailed-positive-rank-one}
\end{equation}
By \eqref{eq:detailed-cj-bound},
\(0\leq c_j/\delta\leq\vartheta_\chi\), so
\eqref{eq:detailed-chi-margin} gives
\[
 \lambda\!\left(B-\frac{c_j}{\delta}I\right)\in\Gamma.
\]
Moreover, \(A_v\) has spectrum in \(\Gamma\) by
\eqref{eq:detailed-v-exact-admissibility}, while $T_j>0$ has spectrum
in $\Gamma_n\subset\Gamma$.  Since the spectral matrix cone
$\mathcal C_\Gamma$ is a convex cone, the first three terms in
\eqref{eq:detailed-decomposition} have a sum in $\mathcal C_\Gamma$.
Its monotonicity under positive-semidefinite addition, recorded in
\eqref{eq:spectral-cone-monotonicity}, therefore gives
\begin{equation}
 \lambda(A_{P_j})\in\Gamma.
 \label{eq:detailed-test-admissible}
\end{equation}
This verifies the admissibility hypothesis needed in the constrained
viscosity supersolution condition.

Let
\[
 G_j=D F(A_{P_j}).
\]
Strict ellipticity makes \(G_j\) positive definite.  Lemma~\ref{lem:chengxu-pairing} gives
\[
 \operatorname{tr}(G_jA_v)\geq0,
 \qquad
 \operatorname{tr}\!\left[
 G_j\left(B-\frac{c_j}{\delta}I\right)
 \right]\geq0,
 \qquad
 \operatorname{tr}(G_jR_j^+)\geq0.
\]
  Taking the trace of
\eqref{eq:detailed-decomposition} against \(G_j\) therefore yields
\begin{equation}
 c_j\operatorname{tr}(G_j T_j)
 \leq\operatorname{tr}(G_jA_{P_j}).
 \label{eq:detailed-trace-first}
\end{equation}
Lemma~\ref{lem:concavity-pairing} and the
lower-test viscosity inequality give
\begin{equation}
 \operatorname{tr}(G_jA_{P_j})
 \leq F(A_{P_j})\leq e^{G(x_j)}.
 \label{eq:detailed-viscosity-upper}
\end{equation}

Apply the arithmetic--geometric mean inequality to the positive matrix
\(G_j^{1/2}T_jG_j^{1/2}\).  Since
\(\lambda(A_{P_j})\in\Gamma\) by
\eqref{eq:detailed-test-admissible}, Assumption~\ref{ass3} applies
\emph{directly} at \(A_{P_j}\) and gives
\(\det G_j\geq\gamma_0\).  Thus, using
\eqref{eq:detailed-MA}--\eqref{eq:detailed-b-lower}, we obtain
\begin{equation}
\begin{aligned}
 \operatorname{tr}(G_jT_j)
 &\geq n\{\det G_j\det T_j\}^{1/n}\\
 &\geq n \gamma_0^{1/n}e^{-C_{\mathrm{MA}}/n}
        \rho_j(x_j)^{1/n}.
 \label{eq:detailed-AMGM}
\end{aligned}
\end{equation}
Because \(\widehat\rho_j\geq we^{nG}\),
\begin{equation}
 \rho_j(x_j)^{1/n}
 \geq
 \frac{w(x_j)^{1/n}e^{G(x_j)}}{\mathcal A_j^{1/n}}.
 \label{eq:detailed-density-lower}
\end{equation}
Combining \eqref{eq:detailed-trace-first}--
\eqref{eq:detailed-density-lower}, and cancelling the strictly positive
factor \(e^{G(x_j)}\), gives
\begin{equation}
 c_j n \gamma_0^{1/n}
 e^{-C_{\mathrm{MA}}/n}
 \frac{w(x_j)^{1/n}}{\mathcal A_j^{1/n}}
 \leq1.
 \label{eq:detailed-cancel-expG}
\end{equation}

On the other hand, positivity of the maximum gives
\begin{equation}
 \varepsilon_jw(x_j)>R_j(x_j)^\alpha.
 \label{eq:detailed-positive-max}
\end{equation}
Since \(\alpha/n=1/(n+1)\), substituting
\eqref{eq:detailed-cj} and \eqref{eq:detailed-positive-max} into the
left-hand side of \eqref{eq:detailed-cancel-expG} gives the strict lower
bound
\begin{align*}
 c_j n \gamma_0^{1/n}
 e^{-C_{\mathrm{MA}}/n}
 \frac{w(x_j)^{1/n}}{\mathcal A_j^{1/n}}
 &>
 \alpha n \gamma_0^{1/n}e^{-C_{\mathrm{MA}}/n}
 \varepsilon_j^{-(n+1)/n}\mathcal A_j^{-1/n}\\
 &=K_\kappa^{-1}
   (\varepsilon_j^{n+1}\mathcal A_j)^{-1/n}
 =2^{1/n}>1,
\end{align*}
contradicting \eqref{eq:detailed-cancel-expG}.  This proves
\eqref{eq:detailed-theta}.

It follows that
\[
 \varepsilon_jw\leq R_j^\alpha.
\]
Raising this inequality to the power \(1/\alpha=(n+1)/n\), and using
the definition of \(\varepsilon_j\), yields
\begin{equation}
 2^{-1/n}K_\kappa^{-1}
 \frac{w^{(n+1)/n}}{\mathcal A_j^{1/n}}
 \leq -\psi_j+\Lambda_j.
 \label{eq:detailed-pointwise-trudinger}
\end{equation}
Moreover,
\begin{equation}
 \Lambda_j
 =(10\alpha)^{n+1}\vartheta_\chi^{-(n+1)}
   \delta^{-(n+1)}
   \varepsilon_j^{-(n+1)}
 =2(10\alpha)^{n+1}K_\kappa^n
   \vartheta_\chi^{-(n+1)}
   \delta^{-(n+1)}\mathcal A_j.
 \label{eq:detailed-Lambda-scale}
\end{equation}

The Hermitian exponential-integrability lemma, cf. Lemma 4.10 in \cite{ChengXu2025},
provides \(a_\omega>0\) and \(C_\omega\) such that
\begin{equation}
 \int_Me^{-a_\omega\psi}\,dV\leq C_\omega
 \quad\text{whenever}\quad
 \omega+\ddc\psi>0,\qquad\sup_M\psi=0.
 \label{eq:detailed-alpha-invariant}
\end{equation}
Choose \(\beta_0>0\) so that
\(2^{1/n}K_\kappa\beta_0\leq a_\omega\).  Equations
\eqref{eq:detailed-pointwise-trudinger}--
\eqref{eq:detailed-alpha-invariant} give
\begin{equation}
 \int_M
 \exp\!\left(
 \beta_0\frac{w^{(n+1)/n}}{\mathcal A_j^{1/n}}
 \right)dV
 \leq
 C\exp\!\left(C\delta^{-(n+1)}\mathcal A_j\right).
 \label{eq:detailed-j-trudinger}
\end{equation}
Since \(\mathcal A_j\to\mathcal A>0\), Fatou's lemma proves
\eqref{eq:detailed-trudinger}.  Finally, if
\(\mathcal A_{\delta,s,\kappa}=0\), then
\(w_{\delta,s}e^{nG}=0\) almost
everywhere.  Since \(e^G>0\) and \(w_{\delta,s}\) is continuous, one has
\(w_{\delta,s}\equiv0\).
\end{proof}

We next estimate the upper level set
$\{(1-\delta)v-u-s>0\}$ in order to apply De Giorgi's lemma.
\begin{lemma}\label{l3.3}
Let $u$ and $v$ be as in
Proposition~\ref{prop:detailed-viscosity-trudinger}, and assume
$e^{nG}\in L^{p_0}(dV)$ for some $p_0>1$.  Fix
$0<\delta<1$, $s_0\geq0$, and $0<\eta<1/n$.  One can choose
$1<\kappa<p_0$, sufficiently close to $1$, so that the following
holds.  If
\[
 \mathcal A_{\delta,s_0,\kappa}\leq\delta^{n+1},
\]
and
\[
 \Omega_{\delta,s}=\{(1-\delta)v-u>s\},
 \qquad
 U_\delta(s)=\int_{\Omega_{\delta,s}}e^{nG}\,dV,
\]
then there is a constant $B_\eta>0$ such that, for every
$s\geq s_0$ and $r>0$,
\begin{equation}\label{eq:weighted-level-recurrence}
 rU_\delta(s+r)\leq B_\eta U_\delta(s)^{1+\eta}.
\end{equation}
The constant is independent of $\delta,s_0,s,r$; it depends only on
the fixed data, $p_0,\eta,\kappa$, and
$\|e^{nG}\|_{L^{p_0}(dV)}$.
\end{lemma}

\begin{proof}
Write
\[
 w_s=((1-\delta)v-u-s)^+,
 \qquad
 \mathcal A_s=\mathcal A_{\delta,s,\kappa}.
\]
Both $w_s$ and $\mathcal A_s$ decrease with $s$.  Hence, for
$s\geq s_0$,
\[
 \delta^{-(n+1)}\mathcal A_s\leq1.
\]
If $\mathcal A_s=0$, Proposition~\ref{prop:detailed-viscosity-trudinger}
gives $w_s\equiv0$, and \eqref{eq:weighted-level-recurrence} is
immediate.  Suppose $\mathcal A_s>0$.  Expanding the exponential in
\eqref{eq:detailed-trudinger} gives, for every finite $R\geq1$,
\begin{equation}\label{eq:stability-moment-bound}
 \|w_s\|_{L^R(dV)}
 \leq C_R\mathcal A_s^{1/(n+1)}.
\end{equation}
Indeed, first take moments of order $R=(n+1)m/n$, $m\in\mathbb N$,
and then use that $dV(M)=1$.

Set
\[
 \theta=\frac{n(1+\eta)}{n+1}\in(0,1)
\]
and choose $Q\in(1,p_0)$ by
\[
 \frac1Q=\theta+\frac{1-\theta}{p_0}.
\]
Choose $1<\kappa<Q$ and define $R<\infty$ by
\[
 \frac1\kappa=\frac1R+\frac1Q.
\]
H\"older interpolation, followed by
\eqref{eq:stability-moment-bound}, yields
\begin{align*}
 \mathcal A_s
 &=\|w_se^{nG}\|_{L^\kappa(dV)}\\
 &\leq\|w_s\|_{L^R(dV)}
       \|e^{nG}\|_{L^Q(\Omega_{\delta,s})}\\
 &\leq C_R\mathcal A_s^{1/(n+1)}
       U_\delta(s)^\theta
       \|e^{nG}\|_{L^{p_0}(dV)}^{1-\theta}.
\end{align*}
Consequently,
\begin{equation}\label{eq:A-by-weighted-level}
 \mathcal A_s\leq C U_\delta(s)^{1+\eta}.
\end{equation}
On $\Omega_{\delta,s+r}$ one has $w_s\geq r$.  Since
$dV(M)=1$ and $\kappa>1$,
\[
 \mathcal A_s
 \geq r\|e^{nG}\mathbf1_{\Omega_{\delta,s+r}}\|_{L^\kappa(dV)}
 \geq rU_\delta(s+r).
\]
Combining this inequality with \eqref{eq:A-by-weighted-level} proves
\eqref{eq:weighted-level-recurrence}.
\end{proof}

\begin{lemma}\label{l3.4}
Let $\phi:[s_0,\infty)\to[0,\infty)$ be nonincreasing and
right-continuous.  Suppose that, for some $\eta>0$ and $B>0$,
\[
 r\phi(s+r)\leq B\phi(s)^{1+\eta}
 \qquad(s\geq s_0,\ r>0).
\]
Then
\begin{equation}\label{eq:de-giorgi-extinction-level}
 \phi(s)=0
 \quad\text{for every}\quad
 s\geq s_0+
 \frac{2B\phi(s_0)^\eta}{1-2^{-\eta}}.
\end{equation}
\end{lemma}

\begin{proof}
If $\phi(s_0)=0$, there is nothing to prove.  Otherwise define
recursively
\[
 s_{j+1}=s_j+2B\phi(s_j)^\eta,
 \qquad s_0\ 
 \text{as in the statement}.
\]
As long as $\phi(s_j)>0$, the assumed recurrence gives
\[
 \phi(s_{j+1})\leq\frac12\phi(s_j),
 \qquad
 \phi(s_j)\leq2^{-j}\phi(s_0).
\]
Therefore $s_j\uparrow s_\infty<\infty$ and
\[
 s_\infty-s_0
 \leq2B\phi(s_0)^\eta
       \sum_{j=0}^\infty2^{-j\eta}
 =\frac{2B\phi(s_0)^\eta}{1-2^{-\eta}}.
\]
For every $s>s_\infty$, monotonicity gives
$0\leq\phi(s)\leq\phi(s_j)\to0$.  Right-continuity gives the same
conclusion at $s=s_\infty$, and
\eqref{eq:de-giorgi-extinction-level} follows.
\end{proof}

\begin{corollary}\label{c3.5}
Let $u,v,\delta,s_0$ be as in Lemma~\ref{l3.3}, and put
$q_0=p_0/(p_0-1)$.  For every
\[
 0<\nu<\frac1{nq_0},
\]
one may choose $\kappa>1$ sufficiently close to $1$ such that, if
$\mathcal A_{\delta,s_0,\kappa}\leq\delta^{n+1}$,
then
\begin{equation}\label{eq:level-to-supremum}
 \sup_M\bigl((1-\delta)v-u\bigr)
 \leq s_0+C_\nu
 \Vol(\Omega_{\delta,s_0})^\nu.
\end{equation}
Here $C_\nu$ is independent of $\delta$ and $s_0$.
\end{corollary}

\begin{proof}
Apply Lemma~\ref{l3.3} with
$\eta=\nu q_0<1/n$
to the nonincreasing, right-continuous function
\[
 U_\delta(s)=\int_{\Omega_{\delta,s}}e^{nG}\,dV.
\]
Lemma~\ref{l3.4} gives
\[
 \sup_M((1-\delta)v-u)
 \leq s_0+C U_\delta(s_0)^\eta.
\]
Here we used the strict positivity of $e^{nG}$ and the continuity of
$(1-\delta)v-u$: if a larger level set were nonempty, it would contain
an open set of positive $U_\delta$-mass.  Finally, H\"older's inequality
gives
\[
 U_\delta(s_0)
 \leq\|e^{nG}\|_{L^{p_0}(dV)}
       \Vol(\Omega_{\delta,s_0})^{1/q_0}.
\]
Since $\eta/q_0=\nu$, this proves
\eqref{eq:level-to-supremum}.
\end{proof}
\begin{lemma}\label{l3.7}
Let $u,v$ be as in Corollary~\ref{c3.5}, and assume in addition that
$v\leq0$ and $v$ is bounded.  Fix
$0<\nu<1/(nq_0)$.
Choose $\kappa>1$ as in that corollary.  If $0<\delta<1$ and
$s_0>0$ satisfy
\begin{equation}\label{eq:one-comparator-threshold-assumptions}
 s_0\geq2\delta\|v\|_{L^\infty(M)},
 \qquad
 \mathcal A_{\delta,s_0,\kappa}\leq\delta^{n+1},
\end{equation}
then
\begin{equation}\label{3.8N}
 \sup_M(v-u)
 \leq s_0+C_\nu s_0^{-\nu}
       \|(v-u)^+\|_{L^1(dV)}^\nu.
\end{equation}
The constant is independent of $\delta$ and $s_0$.
\end{lemma}

\begin{proof}
Since $v\leq0$, we have
\begin{equation}\label{eq:comparator-negative}
 \sup_M(v-u)\leq\sup_M\bigl((1-\delta)v-u\bigr).
\end{equation}
On the other hand,
\begin{align*}
 \Vol(\Omega_{\delta,s_0})
 &\leq \frac{1}{s_0}
   \int_{\Omega_{\delta,s_0}}\bigl((1-\delta)v-u\bigr)^+\,dV \\
 &\leq \frac{1}{s_0}\left(
   \|(v-u)^+\|_{L^1(dV)}
   +\delta\|v\|_{L^\infty(M)}\Vol(\Omega_{\delta,s_0})
   \right).
\end{align*}
Therefore, if $2\delta\|v\|_{\infty}\leq s_0$, then
\begin{equation*}
 \Vol(\Omega_{\delta,s_0})
 \leq \frac{2}{s_0}\|(v-u)^+\|_{L^1(dV)}.
\end{equation*}
Combining this estimate with \eqref{eq:level-to-supremum} and
\eqref{eq:comparator-negative}, we obtain
\begin{equation*}
 \sup_M(v-u)
 \leq s_0+C_\nu s_0^{-\nu}
       \|(v-u)^+\|_{L^1(dV)}^\nu.
\end{equation*}

\end{proof}

\begin{lemma}\label{l3.9}
Under the assumptions of Lemma~\ref{l3.7}, for $0<\delta<1$, define
\begin{equation}\label{eq:threshold-definition}
 s_*(\delta)=
 \inf\left\{
 s\geq2\delta\|v\|_{L^\infty(M)}:
 \mathcal A_{\delta,s,\kappa}\leq\delta^{n+1}
 \right\}.
\end{equation}
For every
\[
 \mu>nq_0
\]
one can choose $\kappa>1$ sufficiently close to $1$ so that
\begin{equation}\label{eq:threshold-bound}
 s_*(\delta)
 \leq
 \max\left\{
 2\delta\|v\|_{L^\infty(M)},
 C_\mu\delta^{-\mu} \|(v-u)^+\|_{L^1(dV)}
 \right\}.
\end{equation}
The constant $C_\mu$ is independent of $\delta$.
\end{lemma}

\begin{proof}
The map $s\mapsto\mathcal A_{\delta,s,\kappa}$ is continuous and
nonincreasing.  Indeed,
\[
 \bigl|\mathcal A_{\delta,s,\kappa}
       -\mathcal A_{\delta,s',\kappa}\bigr|
 \leq |s-s'|\|e^{nG}\|_{L^\kappa(dV)}.
\]
It vanishes for $s\geq\sup_M((1-\delta)v-u)$; hence the set in
\eqref{eq:threshold-definition} is nonempty and its infimum is attained.
Thus either
\[
 s_*(\delta)=2\delta\|v\|_{L^\infty(M)}
\]
or
\begin{equation}\label{eq:threshold-equality}
 \mathcal A_{\delta,s_*(\delta),\kappa}=\delta^{n+1}.
\end{equation}
Only the second case requires proof.

Choose $\kappa<\beta<p_0$.  Since
\[
 \delta^{-(n+1)}
 \mathcal A_{\delta,s_*(\delta),\kappa}=1,
\]
the moment estimate \eqref{eq:stability-moment-bound} is uniform at
$s=s_*(\delta)$.  With
\[
 R=\frac{\kappa\beta}{\beta-\kappa},
 \qquad
 a_\beta=\frac1\beta-\frac1{p_0}>0,
\]
H\"older's inequality gives
\begin{align*}
 \mathcal A_{\delta,s_*,\kappa}
 &\leq
 \|w_{\delta,s_*}\|_{L^R(dV)}
 \|e^{nG}\|_{L^\beta(\Omega_{\delta,s_*})}\\
 &\leq
 C_\beta\mathcal A_{\delta,s_*,\kappa}^{1/(n+1)}
 \|e^{nG}\|_{L^{p_0}(dV)}
 \Vol(\Omega_{\delta,s_*})^{a_\beta}.
\end{align*}
By the proof of Lemma~\ref{l3.7},
\[
 \Vol(\Omega_{\delta,s_*})
 \leq\frac{2\|(v-u)^+\|_{L^1(dV)}}{s_*}.
\]
Using \eqref{eq:threshold-equality}, we obtain
\[
 \delta^n
 \leq C_\beta
 \left(\frac{2\|(v-u)^+\|_{L^1(dV)}}{s_*}\right)^{a_\beta},
\]
and consequently
\begin{equation}\label{eq:threshold-beta-bound}
 s_*(\delta)
 \leq C_\beta
 \delta^{-\frac{n}{a_\beta}} \|(v-u)^+\|_{L^1(dV)}
 =C_\beta
 \delta^{-\frac{np_0\beta}{p_0-\beta}} \|(v-u)^+\|_{L^1(dV)}.
\end{equation}
As $\beta\downarrow1$, the exponent in
\eqref{eq:threshold-beta-bound} decreases to
$np_0/(p_0-1)=nq_0$.
Given $\mu>nq_0$, first choose $\beta>1$ sufficiently close to
$1$ that
$np_0\beta/(p_0-\beta)<\mu$,
and then choose $1<\kappa<\beta$.  Since $0<\delta<1$,
\eqref{eq:threshold-beta-bound} implies
\eqref{eq:threshold-bound}.
\end{proof}

\begin{proposition}
\label{prop:stability}
Suppose that $(f,\Gamma)$ satisfies Assumptions~\ref{ass1} and~\ref{ass3},
and that $\chi$ is the fixed smooth admissible background.
Let $u<0$ be a bounded continuous viscosity solution of
\[
 F(A_u)=e^G,
 \qquad G\in C(M),\qquad e^{nG}\in L^{p_0}(dV),\qquad p_0>1,
\]
and put $q_0=p_0/(p_0-1)$ and $N=nq_0$.
Let $v\in C^2(M)$ be nonpositive and exactly admissible, and suppose
$\|u\|_\infty,\|v\|_\infty\leq B$.
Then, for every $0<\theta<1/(1+N)$,
\begin{equation}\label{eq:direct-l1-stability}
 \sup_M(v-u)\leq C_\theta
 \|(v-u)^+\|_{L^1(dV)}^\theta.
\end{equation}
The constant depends only on the fixed data, $p_0,\theta,B$ and
$\|e^{nG}\|_{L^{p_0}}$; no derivative norm or positive lower bound
of $e^G$ is used.

In particular, if a uniformly bounded family $v_t$ satisfies the same
conditions and
\[
 X_t:=\|(v_t-u)^+\|_{L^1(dV)}\leq C_1t^{\gamma_2},
 \qquad C_1,\gamma_2>0,
\]
then for every $0<\gamma<\gamma_2/(1+N)$,
\begin{equation}\label{eq:quantitative-stability}
 \sup_M(v_t-u)\leq C_2t^\gamma,
 \qquad0<t\leq1.
\end{equation}
The constant has the same dependencies, together with
$\gamma,\gamma_2,C_1$.
\end{proposition}

\begin{proof}
Put $X=\|(v-u)^+\|_{L^1(dV)}$.
If $X=0$, continuity gives $v\leq u$. If $X\geq1$, the scalar
bound $\sup(v-u)\leq2B$ proves \eqref{eq:direct-l1-stability}.
It remains to treat $0<X<1$.

Choose $\mu>N$ and $0<\nu<1/N$ so close to their endpoints that
\begin{equation}\label{eq:stability-exponent-room}
 \theta<\frac1{\mu+1},
 \qquad \theta<\frac{\nu\mu}{\mu+1}.
\end{equation}
This is possible because both right-hand sides tend to $1/(1+N)$
as $\mu\downarrow N$ and $\nu\uparrow1/N$.
Choose $\kappa>1$ sufficiently close to $1$ so that
Lemmas~\ref{l3.7} and~\ref{l3.9} apply simultaneously.
Set
\begin{equation}\label{eq:stability-delta-choice}
 \delta=X^{1/(\mu+1)}\in(0,1).
\end{equation}
Lemma~\ref{l3.9} yields
\[
 s_*(\delta)
 \leq\max\{2\delta\|v\|_\infty,C_\mu\delta^{-\mu}X\}
 \leq C X^{1/(\mu+1)}.
\]
Choose $s_0=C_0X^{1/(\mu+1)}$ with a fixed sufficiently large
$C_0$ so that $s_0\geq s_*(\delta)$ and
$s_0\geq2\delta\|v\|_\infty$.
The definition and attainment of $s_*(\delta)$, and monotonicity of
$s\mapsto\mathcal A_{\delta,s,\kappa}$, imply
$\mathcal A_{\delta,s_0,\kappa}\leq\delta^{n+1}$.
Lemma~\ref{l3.7} now gives
\begin{align*}
 \sup_M(v-u)
 &\leq s_0+C_\nu s_0^{-\nu}X^\nu\\
 &\leq C\left(X^{1/(\mu+1)}+X^{\nu\mu/(\mu+1)}\right)
 \leq C_\theta X^\theta.
\end{align*}
The last inequality follows from \eqref{eq:stability-exponent-room}
and $X<1$. This proves the direct estimate.

For the family conclusion, choose
$\gamma/\gamma_2<\theta<1/(1+N)$ and apply
\eqref{eq:direct-l1-stability}:
\[
 \sup_M(v_t-u)\leq C_\theta C_1^\theta t^{\theta\gamma_2}
 \leq C_2t^\gamma,
 \qquad0<t\leq1.
\]
\end{proof}

\section{Proof of the main theorem}

\begin{proof}[Proof of Theorem~\ref{thm:main}]
The uniform estimate of \cite{GuoPhong2024Subsolutions} gives
\begin{equation}\label{eq:main-oscillation}
 \osc_M u\leq C.
\end{equation}
Normalize $\sup_Mu=0$ during the argument. Theorem~\ref{thm:regularization}
gives a smooth admissible family $u_t$ with uniform scalar bounds.
When applying Proposition~\ref{prop:stability}, shift both $u$ and
$u_t$ by the same fixed constant so that both are negative.
The difference and all admissibility conditions are unchanged.

Fix
\[
 0<\gamma<\frac1{1+nq_0}.
\]
The direct $L^1$ estimate in Proposition~\ref{prop:stability}, together
with \eqref{eq:main-L1-comparator}, gives
\begin{equation}\label{eq:main-uniform-regularization}
 \|u_t-u\|_\infty\leq C_\gamma t^\gamma,
 \qquad0<t\leq1.
\end{equation}
Here we use $\gamma<1$ and $0< t \le 1$ to get $t \le t^{\gamma}$.
By Theorem~\ref{thm:regularization},
\begin{equation}\label{eq:main-laplacian-bound}
 \|\Delta_\omega^{\mathrm{Ch}}u_t\|_\infty\leq Ct^{-1}.
\end{equation}
Apply Lemma~\ref{lem:dyadic-holder} with $b=1$ to obtain
\[
 u\in C^{\alpha_1}(M),
 \qquad\alpha_1=\frac{2\gamma}{1+\gamma}.
\]

Now suppose $u\in C^\alpha(M)$ with $0<\alpha<2\gamma$.
Lemma \ref{prop:smoothing} gives
\[
 \|\nabla h_t\|_\infty^2\leq Ct^{-(1-\alpha)},
 \qquad
 \|\nabla^2h_t\|_\infty\leq Ct^{-(2-\alpha)/2}.
\]
Since $0<t\leq1$ and $1-\alpha\leq(2-\alpha)/2$, the quantity
in \eqref{eq:Dh} satisfies $D_{h_t}\leq Ct^{-(2-\alpha)/2}$.
Theorem~\ref{thm:penalized-c2} therefore improves
\eqref{eq:main-laplacian-bound} to
\begin{equation}\label{eq:main-bootstrapped-laplacian}
 \|\Delta_\omega^{\mathrm{Ch}}u_t\|_\infty
 \leq Ct^{-(2-\alpha)/2}.
\end{equation}
All these bounds are uniform in the independently chosen large
penalty parameter $\beta(t)$; no continuity of this choice in $t$ is
required. Since $\gamma<1/(1+nq_0)<1/2$ and $\alpha<2\gamma<1$,
one has $(2-\alpha)/2>\gamma$. Lemma~\ref{lem:dyadic-holder}, now
with $b=(2-\alpha)/2$, gives
\[
 u\in C^{T_\gamma(\alpha)}(M),
 \qquad
 T_\gamma(\alpha)=\frac{4\gamma}{2\gamma+2-\alpha}.
\]

Set $\alpha_0=0$ and $\alpha_{j+1}=T_\gamma(\alpha_j)$.
The first step above is $\alpha_1=T_\gamma(0)$.
The map $T_\gamma$ is strictly increasing on $[0,2\gamma]$,
and its fixed points are $2\gamma$ and $2$. Moreover,
\[
 T_\gamma(\alpha)-\alpha
 =\frac{(\alpha-2\gamma)(\alpha-2)}
        {2\gamma+2-\alpha}>0
 \qquad(0\leq\alpha<2\gamma).
\]
Consequently $\alpha_j$ increases to $2\gamma$, and finite iteration
proves $u\in C^{\alpha_j}(M)$ for every $j$.

Given $0<\mu<2/(1+nq_0)$, choose
$\mu/2<\gamma<1/(1+nq_0)$ and then a finite $j$ such that
$\mu<\alpha_j$. Compact-manifold H\"older inclusion gives the asserted
$C^\mu$ norm estimate. The choices of $\gamma$ and $j$ depend only on
$n,p_0,\mu$, so the constant has exactly the dependencies in the theorem.
Return to the normalization $\sup_Mu=-1$ to obtain its full norm bound.
\end{proof}

To prove Corollary~\ref{cor:main}, we use an approximation argument that allows
us to apply Theorem~\ref{thm:main}. Because uniqueness of weak solutions to
complex $m$-Hessian equations on Hermitian manifolds is not known for general
$L^p$ right-hand sides, we cannot approximate the given solution directly
using the original equation. Instead, following \cite{LuPhungTo2021}, we use
an auxiliary equation whose right-hand side contains an exponential term.

The following global version of \cite[Corollary~5.2]{KNmeasure} follows from
the localization argument in \cite[Section~9]{KNmeasure}.
\begin{lemma}[Corollary 5.2 \cite{KNmeasure}] \label{lem:plurifine-locality}
Let $u,v \in SH_{\chi,k}(X,\omega)\cap L^{\infty}$. Then
\[
\mathbf{1}_{\{u<v\}}
H_{\chi,k}(\max\{u,v\})
=
\mathbf{1}_{\{u<v\}}
H_{\chi,k}(v)
\]
Consequently,
\[
H_{\chi,k}(\max\{u,v\}) \ge \mathbf{1}_{\{u\ge v\}} H_{\chi,k}(u) + \mathbf{1}_{\{u< v\}} H_{\chi,k}(v)
\]
\end{lemma}

Similarly, we use the following global version of
\cite[Corollary~6.2]{KNmeasure}.
\begin{lemma}\label{lem:local-comparison}
Let $u,v$ be bounded $m$-$\omega$-sh functions in a neighborhood of
$\overline{\Omega}$ such that
\[
\liminf_{z\to\partial\Omega}(u-v)(z)\geq 0.
\]
Assume that $H_{\chi,m}(v)\geq H_{\chi,m}(u)$ in $\Omega$. Then $u\geq v$ on
$\Omega$.
\end{lemma}

\begin{lemma}\label{lem:strict-ratio-domination}
Let $(X,\omega)$ be a compact Hermitian manifold of complex dimension $n$,
and let $\chi$ be a smooth real $(1,1)$-form satisfying
\[
\lambda_{\omega}(\chi(x))\in\Gamma_k
\qquad\text{for every }x\in X.
\]
Suppose that
\[
a,b\in \mathrm{SH}_{\chi,k}(X,\omega)\cap L^\infty(X)
\]
and that, for some constant $0\leq c<1$,
\[
\mathbf{1}_{\{a<b\}}H_{\chi,k}(a)
\leq
c\,\mathbf{1}_{\{a<b\}}H_{\chi,k}(b).
\]
Then
\[
a\geq b
\qquad\text{on }X.
\]
\end{lemma}

\begin{proof}
Adding the same constant to $a$ and $b$ changes neither their Hessian
measures nor the set $\{a<b\}$. We may therefore assume that
\[
b\geq 1 \qquad\text{on }X.
\]

Fix $t\in(0,1)$ such that
\[
c<t^k<1.
\]
We claim that
\[
a\geq tb.
\]
Suppose otherwise. Then
\[
m_t:=\inf_X(a-tb)<0.
\]
Choose a sequence $x_j\to x_0$ such that
\[
a(x_j)-tb(x_j)\longrightarrow m_t.
\]

Choose a coordinate ball
\[
B=\{|z|<r\}
\]
centered at $x_0$, so that $z(x_0)=0$. Since $\chi$ is strictly
$k$-positive, after choosing $\varepsilon_0>0$ sufficiently small, the
function
\[
\zeta(z):=\varepsilon_0\bigl(|z|^2-r^2\bigr)
\]
satisfies
\[
\zeta=0\quad\text{on }\partial B,
\qquad
\zeta<0\quad\text{in }B,
\]
as well as
\[
-\zeta\leq b,
\qquad
\chi-dd^c\zeta\in\Gamma_k
\quad\text{on a neighborhood of }\overline B.
\]
Define
\[
\phi:=tb-(1-t)\zeta.
\]
Because $-\zeta\leq b$, we have
\[
\phi
=
tb+(1-t)(-\zeta)
\leq tb+(1-t)b=b.
\]
Moreover,
\[
\chi+dd^c\phi
=
t(\chi+dd^cb)+(1-t)(\chi-dd^c\zeta).
\]
Both forms on the right are $k$-positive. Hence positivity of the mixed
Hessian measures gives
\[
\begin{aligned}
H_{\chi,k}(\phi)
&=
\left[
t(\chi+dd^cb)+(1-t)(\chi-dd^c\zeta)
\right]^k\wedge\omega^{n-k} \\
&\geq
t^k(\chi+dd^cb)^k\wedge\omega^{n-k} \\
&=
t^kH_{\chi,k}(b).
\end{aligned}
\]

Set
\[
E:=\{a<\phi\}.
\]
Since $\phi\leq b$, we have
\[
E\subset\{a<b\}.
\]
Consequently,
\[
\begin{aligned}
\mathbf{1}_E H_{\chi,k}(a)
&\leq
c\,\mathbf{1}_E H_{\chi,k}(b) \\
&\leq
\frac{c}{t^k}\,
\mathbf{1}_E H_{\chi,k}(\phi) \\
&\leq
\mathbf{1}_E H_{\chi,k}(\phi).
\end{aligned}
\]

Let
\[
w:=\max\{a,\phi\}.
\]
On $E=\{a<\phi\}$, Lemma~\ref{lem:plurifine-locality} gives
\[
\mathbf{1}_E H_{\chi,k}(w)
=
\mathbf{1}_E H_{\chi,k}(\phi)
\geq
\mathbf{1}_E H_{\chi,k}(a).
\]
On the contact set
\[
\{w=a\}=\{a\geq\phi\},
\]
Lemma~\ref{lem:plurifine-locality}, applied to $a\leq w$, gives
\[
\mathbf{1}_{\{w=a\}}H_{\chi,k}(a)
\leq
\mathbf{1}_{\{w=a\}}H_{\chi,k}(w).
\]
Therefore,
\[
H_{\chi,k}(w)\geq H_{\chi,k}(a)
\qquad\text{in }B.
\]

On $\partial B$, we have $\zeta=0$ and hence $\phi=tb$. Thus
\[
a-w
=
a-\max\{a,tb\}
=
\min\{0,a-tb\}
\geq m_t.
\]
Equivalently,
\[
a-m_t\geq w
\qquad\text{on }\partial B.
\]
Since adding a constant does not change the Hessian measure,
\[
H_{\chi,k}(a-m_t)
=
H_{\chi,k}(a)
\leq
H_{\chi,k}(w).
\]
Lemma~\ref{lem:local-comparison} therefore yields
\[
a-m_t\geq w
\qquad\text{throughout }B,
\]
or equivalently,
\[
a-w\geq m_t
\qquad\text{in }B.
\]

On the other hand,
\[
\begin{aligned}
a(x_j)-\phi(x_j)
&=
a(x_j)-tb(x_j)+(1-t)\zeta(x_j) \\
&\longrightarrow
m_t+(1-t)\zeta(x_0).
\end{aligned}
\]
Since
\[
\zeta(x_0)=-\varepsilon_0r^2<0,
\]
we obtain
\[
m_t+(1-t)\zeta(x_0)<m_t.
\]
In particular, for all sufficiently large $j$, one has
$a(x_j)<\phi(x_j)$, and hence
\[
a(x_j)-w(x_j)
=
a(x_j)-\phi(x_j)<m_t,
\]
contradicting $a-w\geq m_t$ in $B$.

Thus $m_t\geq0$, and consequently
\[
a\geq tb.
\]
This holds for every $t\in(0,1)$ satisfying $c<t^k$. Letting
$t\uparrow1$ gives
\[
a\geq b.
\]
\end{proof}

\begin{proposition}\label{prop:exponential-uniqueness}
Let $u,v\in SH_{\chi,k}(X,\omega)\cap L^{\infty}(X)$ satisfy
\begin{equation*}
 H_{\chi,k}(u)=e^u h\,\omega^n,
 \qquad
 H_{\chi,k}(v)=e^v h\,\omega^n.
\end{equation*}
Then $u=v$.
\end{proposition}
\begin{proof}
For every $\epsilon>0$,
\begin{equation*}
\begin{split}
& \mathbf{1}_{\{u<v-\epsilon\}} (\chi+ dd^c u)^k \wedge \omega^{n-k}= \mathbf{1}_{\{u<v-\epsilon\}}  e^u h \omega^n \le \mathbf{1}_{\{u<v-\epsilon\}}  e^{v-\epsilon} h \omega^n \\
& = \mathbf{1}_{\{u<v-\epsilon\}} e^{-\epsilon} (\chi+ dd^c (v-\epsilon))^k \wedge \omega^{n-k}. 
\end{split}
\end{equation*}
Lemma~\ref{lem:strict-ratio-domination} gives
\begin{equation*}
u \ge v-\epsilon.
\end{equation*}
Letting $\epsilon$ tend to zero, we obtain $u\geq v$. The reverse inequality
follows similarly.
\end{proof}

\begin{proposition}\label{thm:noncollapse-c0}
Let \((X,\omega)\) be a connected compact Hermitian manifold of complex
dimension \(n\), let \(1\leq k\leq n\), and let \(\chi\) be a smooth real
\((1,1)\)-form satisfying
\[
  \lambda[\chi](x)\in \Gammak
  \qquad\text{for every }x\in X.
\]
Let \(g>0\) be smooth, and let \(u\in C^\infty(X)\) be
\(\Gammak\)-admissible and solve
\begin{equation}
  (\chi+\ddc u)^k\wedge\omega^{n-k}
  =e^u g\,\omega^n.
  \label{eq:monotone-k-hessian}
\end{equation}
Assume that
\begin{equation}
  p> \frac{n}{k},
  \qquad
  \|g\|_{L^p(X,\omega^n)}\leq B,
  \qquad
  \int_X g^{1/k}\,\omega^n\geq a>0.
  \label{eq:noncollapse-assumptions}
\end{equation}
Then
\[
  \sup_X u\leq C,
\]
where \(C\) depends only on
\[
  (X,\omega),\quad \chi,\quad n,\quad k,\quad p,\quad B,\quad a,
\]
\end{proposition}

\begin{proof}
Define
\[
  M:=\sup_Xu,
  \qquad
  v:=u-M.
\]
Then \(\sup_Xv=0\), \(v\leq0\), and
\[
  \chi+\ddc v=\chi+\ddc u.
\]
Since \(\Gammak\subset\Gamma_1\), admissibility implies
\begin{equation}
  \Delta_\omega^{\Ch}v
  =\tr_\omega(\ddc v)
  \geq-\tr_\omega\chi.
  \label{eq:chern-laplacian-lower}
\end{equation}
The Green-kernel estimate for the fixed Chern Laplacian, together with
\(\sup_Xv=0\), gives
\begin{equation}
  \|v\|_{L^q(X,\omega^n)}\leq C_q
  \qquad
  \text{for every }1\leq q<\frac{n}{n-1}.
  \label{eq:green-Lq}
\end{equation}
When \(n=1\), the upper endpoint in \eqref{eq:green-Lq} is interpreted as
\(+\infty\).  Indeed, if
\[
  \mu:=\Delta_\omega^{\Ch}v+\tr_\omega\chi\geq0,
\]
then the Gauduchon adjoint density fixes the total mass of \(\mu\), and the
Green kernel has the usual real \(2n\)-dimensional singularity
\(d(x,y)^{2-2n}\), logarithmic when \(n=1\).  This proves
\eqref{eq:green-Lq}.

Let \(\rho>0\) be the fixed Gauduchon density, normalized once and for all,
such that
\begin{equation}
  \ddc(\rho\,\omega^{n-1})=0.
  \label{eq:noncollapse-gauduchon-density}
\end{equation}
For a \(\Gammak\)-admissible real \((1,1)\)-form \(\alpha\), the normalized
Maclaurin inequality gives
\begin{equation}
  \frac{\alpha\wedge\omega^{n-1}}{\omega^n}
  \geq
  \left(
    \frac{\alpha^k\wedge\omega^{n-k}}{\omega^n}
  \right)^{1/k}.
  \label{eq:maclaurin}
\end{equation}
Applying \eqref{eq:maclaurin} to
\(\alpha=\chi+\ddc v\), and using
\eqref{eq:monotone-k-hessian}, yields
\begin{equation}
  (\chi+\ddc v)\wedge\omega^{n-1}
  \geq e^{(M+v)/k}g^{1/k}\omega^n.
  \label{eq:maclaurin-applied}
\end{equation}
Multiplying by \(\rho\) and integrating, we obtain
\begin{align}
  T_\chi
  &:=
  \int_X\rho\,\chi\wedge\omega^{n-1}
  \notag\\
  &=
  \int_X\rho(\chi+\ddc v)\wedge\omega^{n-1}
  \notag\\
  &\geq
  e^{M/k}\int_X\rho e^{v/k}g^{1/k}\omega^n.
  \label{eq:gauduchon-maclaurin}
\end{align}
The equality in \eqref{eq:gauduchon-maclaurin} follows from
\eqref{eq:noncollapse-gauduchon-density} and Stokes' theorem:
\[
  \int_X\rho\,\ddc v\wedge\omega^{n-1}
  =
  \int_Xv\,\ddc(\rho\,\omega^{n-1})
  =0.
\]
Moreover, \(T_\chi>0\), since
\(\lambda_\omega(\chi)\in\Gammak\subset\Gamma_1\).

Set
\[
  A:=\int_X\rho g^{1/k}\omega^n.
\]
Since \(\rho>0\) is fixed, the noncollapse condition gives
\begin{equation}
  A\geq(\min_X\rho)a=:a_\rho>0.
  \label{eq:A-lower}
\end{equation}
Consider the probability measure
\[
  d\mu
  :=\frac{\rho g^{1/k}\omega^n}{A}.
\]
By Jensen's inequality,
\begin{align}
  \int_X\rho e^{v/k}g^{1/k}\omega^n
  &=
  A\int_Xe^{v/k}\,d\mu
  \notag\\
  &\geq
  A\exp\left(
    \frac{1}{kA}
    \int_Xv\rho g^{1/k}\omega^n
  \right).
  \label{eq:jensen}
\end{align}
Combining \eqref{eq:gauduchon-maclaurin} and \eqref{eq:jensen}, and taking
logarithms, gives
\begin{equation}
  M
  \leq
  k\log\frac{T_\chi}{A}
  +\frac{1}{A}
   \int_X(-v)\rho g^{1/k}\omega^n.
  \label{eq:M-upper-pre}
\end{equation}

Let
\[
  r:=kp,
  \qquad
  q:=\frac{r}{r-1}=\frac{kp}{kp-1}.
\]
Since \(kp>n\), one has
\[
  q<\frac{n}{n-1},
\]
so \eqref{eq:green-Lq} applies. H\"older's inequality gives
\begin{align}
  \int_X(-v)\rho g^{1/k}\omega^n
  &\leq
  \|v\|_{L^q}
  \|\rho g^{1/k}\|_{L^{kp}}
  \notag\\
  &\leq
  C_q\|g\|_{L^p}^{1/k}
  \leq C_qB^{1/k}.
  \label{eq:weighted-holder}
\end{align}
It follows from \eqref{eq:A-lower}, \eqref{eq:M-upper-pre}, and
\eqref{eq:weighted-holder} that
\begin{equation}
  M\leq
  M_+
  :=
  k\log\frac{T_\chi}{a_\rho}
  +\frac{C_qB^{1/k}}{a_\rho}.
  \label{eq:M-upper}
\end{equation}
\end{proof}

\begin{lemma}
\label{lem:uniform-hessian-continuity}
Let $z_j,z\in SH_{\chi,k}(M,\omega)\cap L^\infty(M)$ and suppose that
$z_j\to z$ uniformly on $M$. Then
\[
 H_{\chi,k}(z_j)\rightharpoonup H_{\chi,k}(z)
\]
weakly as Radon measures.
\end{lemma}

\begin{proof}
This is the uniform-convergence case of the weak-continuity argument in
\cite[proof of Lemma~5.1]{KNmeasure}.
\end{proof}

We are now ready to prove Corollary~\ref{cor:main}.
\begin{proof}[Proof of Corollary~\ref{cor:main}]
Fix
\[
 0<\mu<\alpha_*:=
  \frac{2(kp_0-n)}{kp_0(n+1)-n},
\]
and choose $\alpha$ with $\mu<\alpha<\alpha_*$. Since
\eqref{eq:main-cor} is invariant under adding constants to $u$, we normalize
$\sup_M u=0$. By \cite{GuoPhong2024Subsolutions}, there is a constant
$C_1$, depending only on the fixed data and $\|g\|_{L^{p_0}}$, such that
\[
 \|u\|_{L^\infty(M)}\leq C_1.
\]
Set $h=e^{-u}g$. Then $h\in L^{p_0}(M)$ and
$\|h\|_{L^{p_0}}\leq e^{C_1}\|g\|_{L^{p_0}}$.
Moreover, \eqref{eq:main-cor} says that $u$ is a solution of the twisted
equation
\begin{equation}\label{eq:corollary-twisted-equation}
 H_{\chi,k}(w)=e^{w-u}g\,\omega^n,
 \qquad w\in SH_{\chi,k}(M,\omega)\cap L^\infty(M).
\end{equation}

The function $h$ is not identically zero. Indeed, otherwise
$H_{\chi,k}(u)=0$, and Lemma~\ref{lem:strict-ratio-domination}, applied
with $a=u$, $c=0$, and $b$ an arbitrary constant, would give
$u\geq b$ for every $b\in\mathbb R$. We may therefore choose positive
smooth functions $h_l$ converging to $h$ in $L^{p_0}(M)$ and constants
$B,a>0$ such that
\begin{equation}\label{eq:uniform-hl}
 \|h_l\|_{L^{p_0}(M)}\leq B,
 \qquad
 \int_M h_l^{1/k}\,\omega^n\geq a
 \quad\text{for every }l.
\end{equation}
By \cite[Theorem~3.1]{ChengXu2025}, there is a smooth admissible solution
$u_l$ of
\begin{equation}\label{eq:corollary-approximating-equation}
 H_{\chi,k}(u_l)=e^{u_l}h_l\,\omega^n.
\end{equation}
Proposition~\ref{thm:noncollapse-c0} and \eqref{eq:uniform-hl} give a
constant $C_2$, independent of $l$, such that
\begin{equation}\label{eq:corollary-upper-bound}
 \sup_M u_l\leq C_2.
\end{equation}
In particular,
\begin{equation}\label{eq:corollary-rhs-lp}
 \|e^{u_l}h_l\|_{L^{p_0}(M)}\leq e^{C_2}B.
\end{equation}

Put $r_0=p_0k/n>1$. For the degree-one operator
\[
 f_k(\lambda)=
 \left(\binom nk^{-1}\sigma_k(\lambda)\right)^{1/k},
\]
the scalar quantity $e^{nG_l}$ in Theorem~\ref{thm:main} is
$(e^{u_l}h_l)^{n/k}$. Hence \eqref{eq:corollary-rhs-lp} gives a uniform
$L^{r_0}$ bound for $e^{nG_l}$. Applying Theorem~\ref{thm:main} to
$u_l-\sup_M u_l-1$ therefore yields a constant $C_4$, independent of $l$,
such that
\begin{equation}\label{eq:corollary-normalized-holder}
 \|u_l-\sup_M u_l\|_{C^\alpha(M)}\leq C_4.
\end{equation}

We next claim that $\sup_M u_l$ is bounded from below uniformly in $l$.
Otherwise, after passing to a subsequence, we would have
$m_l:=\sup_M u_l\to-\infty$. Set $v_l=u_l-m_l$. By
\eqref{eq:corollary-normalized-holder} and compactness, a further
subsequence converges uniformly to a bounded function
$v_\infty\in SH_{\chi,k}(M,\omega)$. From
\eqref{eq:corollary-approximating-equation},
\[
 H_{\chi,k}(v_l)=e^{m_l}e^{v_l}h_l\,\omega^n,
\]
and the right-hand side tends to zero in $L^1(M)$ because $v_l\leq0$,
$m_l\to-\infty$, and the $h_l$ are uniformly bounded in $L^{p_0}$.
Lemma~\ref{lem:uniform-hessian-continuity} gives
\[
 H_{\chi,k}(v_l)\rightharpoonup H_{\chi,k}(v_\infty).
\]
Thus $H_{\chi,k}(v_\infty)=0$. Applying
Lemma~\ref{lem:strict-ratio-domination} with $a=v_\infty$, $c=0$, and
$b$ an arbitrary constant would imply $v_\infty\geq b$ for every
$b\in\mathbb R$, a contradiction. This proves the claim.

Combining the claim, \eqref{eq:corollary-upper-bound}, and
\eqref{eq:corollary-normalized-holder}, we obtain
\begin{equation}\label{eq:corollary-uniform-holder}
 \|u_l\|_{C^\alpha(M)}\leq C_6
\end{equation}
for every $l$. The compact embedding
$C^\alpha(M)\hookrightarrow C^\mu(M)$ now gives, after passing to a
subsequence,
\[
 u_l\longrightarrow u_\infty
 \quad\text{in }C^\mu(M)
\]
for some $u_\infty\in SH_{\chi,k}(M,\omega)\cap C^\alpha(M)$. In addition,
$h_l\to h$ in $L^{p_0}$ and $u_l\to u_\infty$ uniformly, so
\[
 e^{u_l}h_l\longrightarrow e^{u_\infty}h
 \quad\text{in }L^{p_0}(M).
\]
Using Lemma~\ref{lem:uniform-hessian-continuity} in
\eqref{eq:corollary-approximating-equation}, we conclude that
\[
 H_{\chi,k}(u_\infty)
 =e^{u_\infty}h\,\omega^n
 =e^{u_\infty-u}g\,\omega^n.
\]
Thus both $u_\infty$ and $u$ solve
\eqref{eq:corollary-twisted-equation}. Proposition
\ref{prop:exponential-uniqueness} gives $u_\infty=u$. Since the convergence
holds in $C^\mu(M)$, it follows that $u\in C^\mu(M)$, as required.
\end{proof}

\bigskip
\begin{center}
\begin{minipage}{0.88\textwidth}
\small
\textbf{Yulun Xu}\\[0.35em]
University of Toronto, 40 St. George Street, Toronto, ON, Canada\\[0.35em]
\href{mailto:yulun.1.xu@gmail.com}{\texttt{yulun.1.xu@gmail.com}}
\end{minipage}
\end{center}

\end{document}